\documentclass{amsart}
\usepackage[utf8]{inputenc}
\usepackage[T1]{fontenc}
\usepackage[final]{microtype}
\usepackage{hyperref,amsthm,amssymb,interval,tikz,tikz-cd}
\usepackage{xcolor}
\theoremstyle{definition}
\newtheorem{example}[equation]{Example}
\newtheorem{situation}[equation]{}
\newtheorem{remark}[equation]{Remark}
\newtheorem{definition}[equation]{Definition}
\theoremstyle{plain}
\newtheorem{theorem}[equation]{Theorem}
\newtheorem*{theorem*}{Theorem}
\newtheorem{proposition}[equation]{Proposition}
\newtheorem{lemma}[equation]{Lemma}
\newtheorem{corollary}[equation]{Corollary}
\numberwithin{equation}{section}
\usepackage[]{enumitem}
\setlist[enumerate]{labelindent=\parindent,leftmargin=*,topsep=0.4ex,itemsep=0.1ex}
\setlist[itemize]{labelindent=\parindent,leftmargin=*,topsep=0.4ex,itemsep=-1ex,partopsep=1ex,parsep=1ex}
\setlist[enumerate,1]{labelindent=\parindent, leftmargin=*,label=\textup{(\arabic*)},ref=\textup{\arabic*}}
\newcommand{\tube}[1]{\mathop{]#1[}}
\renewcommand{\setminus}{\mathbin{\rule[0.2em]{0.67em}{0.12em}}}%
\renewcommand{\mathbb}{\mathbf}

\title[Exponential sums]{Exponentially twisted de~Rham cohomology and rigid cohomology}
\author{Shizhang Li}
\address{Department of Mathematics,  University of Michigan}
\email{shizhang@umich.edu}
\author{Dingxin Zhang}
\address{Yau Mathematical Sciences Center, Tsinghua University}
\email{dingxin@tsinghua.edu.cn}

\begin{document}
\maketitle

\begin{abstract}
  We prove a comparison theorem between exponentially twisted de~Rham cohomology
  and rigid cohomology with coefficients in a Dwork crystal.
\end{abstract}

% \tableofcontents
\section*{Introduction}

\begin{situation}\label{sit:intro}%
  \textbf{Cohomology groups with exponential twists.}
  Let \(k\) be a field. Let \(f\colon X \to \mathbb{A}^{1}_{k}\) be a morphism
  of algebraic varieties over \(k\). Depending upon what \(k\) is, one can
  consider the following realizations of the ``exponential motive''
  (in the sense of Fresán--Jossen~\cite{fresan-jossen:expmot})
  associated with the function \(f\).

  \begin{enumerate}[wide]
  \item
    Betti realization.
    When \(k\) is the field \(\mathbb{C}\) of complex numbers, one can
    consider the relative singular cohomology
    \(\mathrm{H}^{\bullet}(X^{\mathrm{an}},f^{-1}(t)^{\mathrm{an}})\)
    (here \(t\in \mathbb{C}\) and \(|t|\) is sufficiently large,
    in fact any \emph{typical value}\footnote{We say \(t\) is a typical value of
      \(f\) if \(t\) falls in the largest open subset of \(\mathbb{A}^{1,\mathrm{an}}\) on
      which \(R^{i}f_{\ast}\mathbb{Q}\) are locally constant for all \(i\).} of \(f\) will do).
    The Betti realization has an integral structure.
    \label{item:betti}
  \item
    De~Rham realization.
    For an arbitrary \(k\), one can consider the exponentially
    twisted de~Rham cohomology
    \(\mathrm{H}^{\bullet}_{\mathrm{DR}}(X,\nabla_f)\), where \(\nabla_f\) is
    the integrable connection on the trivial module \(\mathcal{O}_{X}\) defined by
    \(\nabla_f(h) = \mathrm{d}h + h \mathrm{d}f\).
    \label{item:alg-dr}
    It should be brought to the reader's attention that the connection
    \(\nabla_{f}\)  has irregular singularity, thus does not fit into the
    picture of~\cite{deligne:regular-singularity-differential-equation}.

    When \(k = \mathbb{C}\), and when the Betti cohomology is taken
    \(\mathbb{C}\) as its
    coefficient ring, it is known that the cohomology groups in (\ref{item:betti}) and
    (\ref{item:alg-dr}) are isomorphic. This theorem could be attributed to
    Deligne--Malgrange~\cite[pp.~79,~81,~87]{deligne-malgrange-ramis:irregular-correspondence},
    Dimca--Saito~(the upshot is
    \cite[Proposition~2.8]{dimca-saito_morihiko:cohomology-general-fiber-polynomial-map}),
    and Sabbah~\cite{sabbah:comparison-theorem-elementary-irregular-d-modules}.

  \item Rigid analytic de Rham realization.
    When \(k\) is a field equipped with a complete ultrametric, one can
    consider the rigid analytic version of the twisted de~Rham cohomology
    \(\mathrm{H}_{\mathrm{DR}}^{\bullet}(X^{\mathrm{an}},\nabla_{f})\).
    \label{item:rig-an-dr}

    When \(k\) is of characteristic \(0\),
    % and when \(f\) is ``in a general position'' with respect to \(p\),
    it follows from the André--Baldassarri comparison
    theorem~\cite[Theorem~6.1]{andre-baldassarri:de-rham-cohomology-differential-modules-on-algebraic-varieties}
    that (\ref{item:alg-dr}) and (\ref{item:rig-an-dr}) are isomorphic%
    \footnote{%
      The André--Baldassarri theorem as stated in loc.~cit.~does not immediately
      imply the said isomorphy, as the variety we consider is not assumed to be
      defined over a number field. Instead of walking through their d\'evissage
      argument, we shall present a proof of it.}.
    Note that the complex analytic version of this result is false even in the
    simplest situation \(X = \mathbb{A}^1\), \(f= \mathrm{Id}\), since
    \(\nabla_f\) has irregular singularity at infinity (indeed, the complex
    analytification of \(\nabla_f\) is isomorphic to the trivial connection:
    \(\nabla_{f}^{\mathrm{an}}=\mathrm{e}^{-f} \circ \mathrm{d} \circ \mathrm{e}^{f}\)).

  \item \(\ell\)-adic realization.
    Assume that \(k\) is a finite field of characteristic \(p>0\).
    Fix a nontrivial additive character
    \(\psi\colon k \to \mathbb{C}^{\ast}\), and an algebraic closure
    \(k^{\mathrm{a}}\) of \(k\). Let \(k_m\) be the subfield of \(k^{\mathrm{a}}\)
    such that \([k_m:k]=m\). One can consider the L-series associated with the
    exponential sums defined by \(f\):
    \begin{equation*}
      S_{m}(f) = \sum_{x \in X(k_m)} \psi(\operatorname{Tr}_{k_m/k}f(x)); \quad
      L_f(t) = \exp\left\{ \sum_{m=1}^{\infty} \frac{S_m}{m}t^{m} \right\}.
    \end{equation*}
    By a theorem of
    Grothendieck, this L-series is the (super) determinant of the Frobenius
    operation on a twisted \(\overline{\mathbb{Q}}_{\ell}\)-étale cohomology
    theory.
    \label{item:et}
  \item Crystalline realization.
    When \(k\) is a perfect field of characteristic \(p > 0\), one can consider the
    rigid cohomology (or rigid cohomology with compact support)
    \(\mathrm{H}^{\bullet}_{\mathrm{rig}}(X/K,f^{\ast}\mathcal{L}_{\pi})\). Here,
    \(\mathcal{L}_{\pi}\) is a certain overconvergent isocrystal on
    \(\mathbb{A}^{1}_k\) called ``Dwork isocrystal''.

    By a theorem of Berthelot, the rigid cohomology admits a Frobenius operation
    which, when \(k\) is finite, could determine the L-series as in
    Item (\ref{item:et}).
    \label{item:rigid-coh}
  \end{enumerate}
\end{situation}
% These objects are all extensively studied by specialists.
In this note, we shall prove a comparison theorem between
(\ref{sit:intro}/\ref{item:alg-dr})  and (\ref{sit:intro}/\ref{item:rigid-coh}),
thus building a bridge between topology and arithmetic.
% We emphasize that the comparison theorem we are about to state deals with
% integrable connections with \emph{irregular singularities}. Thus it is not
% directly related to the comparison theorem between e.g., logarithmic de~Rham
% cohomology and rigid cohomology of logarithmic crystals, which are objects with
% \emph{regular singularities}.

To state the result, let us set up some notation.
\begin{itemize}[wide]
\item Let \(X\) be a smooth scheme over a finitely generated
  \(\mathbb{Z}\)-algebra \(R\) of characteristic \(0\) which is an integral
  domain. Let \(f\colon X \to \mathbb{A}^{1}_{R}\) be a morphism.
  Let \(\sigma\colon R \to \mathbb{C}\) be any embedding of \(R\) into the field
  \(\mathbb{C}\) of complex numbers.
\item For each maximal \(\mathfrak{p}\) of \(R\), let \(\kappa(\mathfrak{p})\) be
  the residue field of \(\mathfrak{p}\),
  let \(K_{\mathfrak{p}}\) be the field of
  fractions of \(W(\kappa(\mathfrak{p}))[\zeta_p]\),
  the ring of Witt vectors of \(\kappa(\mathfrak{p})\) with \(p\)\textsuperscript{th}
  roots of unity adjoined.
\item For an \(R\)-algebra \(R'\), we still use \(f\) to denote the morphism
  \(X_{R'} = X \times_R \operatorname{Spec}(R')\to \mathbb{A}^{1}_{R'}\).
\end{itemize}

The most accessible statement of our result is the following.

\begin{theorem}%
  \label{theorem:main-weak}
  There is a dense Zariski open subset \(U\) of
  \(\mathrm{Spec}(R)\) such that for any \(\mathfrak{p} \in U\), any integer \(m\),
  the \(K_{\mathfrak{p}}\)-dimension of the rigid cohomology
  \(\mathrm{H}_{\mathrm{rig}}^{m}(X_{\kappa(\mathfrak{p})}/K_{\mathfrak{p}},f^{\ast}\mathcal{L}_{\pi})\)
  equals the complex dimension of the complex vector space
  \(\mathrm{H}^{m}_{\mathrm{DR}}(X\times_{R,\sigma}\operatorname{Spec}(\mathbb{C}),\nabla_f)\).
\end{theorem}

In the main text, we shall give a precise condition on which \(\mathfrak{p}\) is
good for the comparison theorem to hold based on ramifications of \(f\) at
infinity. See Theorem~\ref{theorem:main}.

We shall also prove a version of Theorem~\ref{theorem:main-weak} comparing the
algebraic Higgs cohomology associated with \(f\) and an overconvergent Higgs
cohomology. See
Proposition~\ref{proposition:overconvergent-higgs-qis-algebraic-higgs}. This
latter comparison theorem  has significant weaker restrictions on the shape of \(f\).

When the \(X\) is an open subspace of \(\mathbb{P}^{1}\), the theorem is due to
Phillepe Robba~\cite{robba:index-p-adic-differential-operators-3}.
When \(X\) is a curve, the theorem is a simple corollary of Joe
Kramer-Miller's
theorem~\cite[Theorem~1.1]{kramer-miller:newton-above-hodge-for-curves}, see
Example~\ref{example:illustration-curve}.

Theorem~\ref{theorem:main-weak} (or rather its stronger version,
Theorem~\ref{theorem:main}) is desirable, because it seems that in the
literature, the methods used to
study exponential sums are either toric in nature, or only applicable to
``tame'' functions (e.g., Newton-nondegenerate, convenient Laurent polynomials),
whereas Theorem~\ref{theorem:main} is unconditional.
In practice, Theorem~\ref{theorem:main} allows us to calculate the dimension of the
rigid cohomology, hence the degree of the L-series, using
topological methods. See \S\ref{sec:examples} for some concrete examples. Here we only
explain one general procedure for producing examples on which
Theorem~\ref{theorem:main-weak} is applicable.

\begin{example}[``Standard situation'']%
  \label{example:how-to-produce-such-a-function}
  Let \(P\) be a smooth projective variety of pure dimension \(n\) over a number
  field \(\mathbf{K}\).
  Let \(\mathcal{L}_{1},\ldots,\mathcal{L}_{r}\) be invertible sheaves on \(P\).
  Suppose we are given sections \(s_{i} \in \mathrm{H}^{0}(X,\mathcal{L}_{i})\)
  of these invertible sheaves such that the zero loci \(D_{i} = \{s_i=0\}\) form
  a divisor with strict normal crossings.

  Let \(X = P \setminus \bigcup_{i=1}^{r}D_{i}\).
  Let \(s_0 \in \mathrm{H}^{0}(X,\mathcal{L}_{1}^{e_1} \otimes \cdots \otimes \mathcal{L}_{r}^{e_r})\),
  and \(s_{\infty} = s_{1}^{e_1} \cdots s_{r}^{e_r}\).
  Assume that \(X_0=\{s_0=0\}\)
  is a smooth subvariety of \(P\). Write \(X_{\infty} = \{s_{\infty}=0\}\) be
  the vanishing divisor of \(s_{\infty}\).

  Then \(g(x) = s_0(x)/s_{\infty}(x)\) is a well-defined regular function on
  \(X\). We assume in addition the divisor \(X_0=\{s_0 = 0\}\) is transverse to all the
  intersections \(D_{i_1} \cap \cdots \cap D_{i_m}\).

  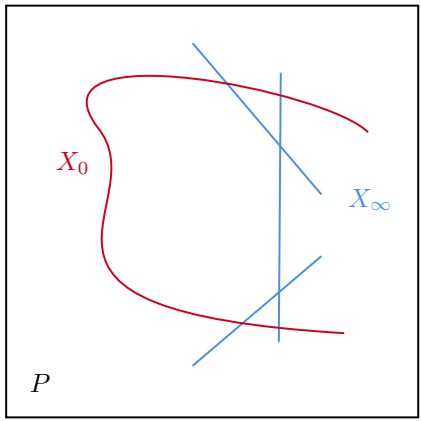
\begin{figure*}[ht]
    \centering
    \tikzset{every picture/.style={line width=0.75pt}} %set default line width to 0.75pt
    \begin{tikzpicture}[x=0.75pt,y=0.75pt,yscale=-1,xscale=1]
      % uncomment if require: \path (0,300); %set diagram left start at 0, and has height of 300
      % Shape: Square [id:dp17384887906614588]
      \draw   (217.35,30.08) -- (425.17,30.08) -- (425.17,237.9) -- (217.35,237.9) -- cycle ;
      % Straight Lines [id:da9246700351422754]
      \draw [color={rgb, 255:red, 74; green, 144; blue, 226 }  ,draw opacity=1 ]   (311.36,49) -- (376.35,125.4) ;
      % Straight Lines [id:da35133979430353235]
      \draw [color={rgb, 255:red, 74; green, 144; blue, 226 }  ,draw opacity=1 ]   (355.85,63.9) -- (354.85,199.9) ;
      % Straight Lines [id:da4504162139311847]
      \draw [color={rgb, 255:red, 74; green, 144; blue, 226 }  ,draw opacity=1 ]   (311.35,211.9) -- (376.36,156.5) ;
      % Curve Lines [id:da7076108392625102]
      \draw [color={rgb, 255:red, 208; green, 2; blue, 27 }  ,draw opacity=1 ]   (387.85,195.4) .. controls (196.85,183.9) and (295.35,129.9) .. (263.35,91.4) .. controls (227.85,43.4) and (379.35,72.9) .. (399.86,94) ;

      % Text Node
      \draw (240.86,102.4) node [anchor=north west][inner sep=0.75pt]  [color={rgb, 255:red, 208; green, 2; blue, 27 }  ,opacity=1 ]  {$X_{0}$};
      % Text Node
      \draw (388.36,120.9) node [anchor=north west][inner sep=0.75pt]  [color={rgb, 255:red, 74; green, 144; blue, 226 }  ,opacity=1 ]  {$X_{\infty}$};
      % Text Node
      \draw (227.86,214.4) node [anchor=north west][inner sep=0.75pt]    {$P$};
    \end{tikzpicture}
    \caption{Standard situation}
    \label{fig:standard}
  \end{figure*}

  If \(\mathfrak{p}\) is a prime of \(\mathcal{O}_{\mathbf{K}}\) such that
  \begin{enumerate}
  \item the logarithmic pairs \((P,X_{\infty})\)
    and \((X_0,X_0 \cap X_{\infty})\)
    all have good reductions at \(\mathfrak{p}\), and
  \item the residue characteristic of \(\mathfrak{p}\) does not divide
    \(e_1\cdots e_r\),
  \end{enumerate}
  then Theorem \ref{theorem:main-weak} (for the function \(g\colon X \to \mathbb{A}^{1}\))
  is valid at \(\mathfrak{p}\).
  Moreover, if
  \(\bigotimes_{i=1}^{r}\mathcal{L}^{\otimes e_i}\) is ample, then the rigid
  cohomology is nonzero in degree \(n\) only.

  For more details, we refer the reader to
  Corollary~\ref{corollary:standard-situation}.
\end{example}

\begin{situation}\label{situation:previous}%
  \textbf{Previously known theorems about degrees of L-series.}

  \begin{enumerate}[wide]
  \item[(a)]  \textit{\(\ell\)-adic theorems.}
    When \(k\) is a finite field, and when \(X = \mathbb{G}_{\mathrm{m}}^{n}\),
    Denef and Loeser studied the étale cohomology appeared in (\ref{sit:intro}/\ref{item:et}).
    They showed
    that~\cite[Theorem~1.3]{denef-loeser:weights-exponential-sums-newton-polyhedra}
    if \(f\) is ``nondegenerate with respect to its Newton polyhedron at infinity'',
    then the twisted étale cohomology is acyclic except in degree \(n\), and the Frobenius
    eigenvalues are pure of weight \(n\).
    In general, they were able to show that the Euler
    characteristic of the étale cohomology agrees with the Euler characteristic of
    the algebraic de~Rham cohomology (\ref{sit:intro}/\ref{item:alg-dr}) defined by a Teichmüller
    lift of \(f\) (the combinatorial formulas for both theories match).

    In the standard situation~\eqref{example:how-to-produce-such-a-function},
    assuming the invertible sheaves \(\mathcal{L}_{i}\) are ample,
    the étale cohomology associated with the exponential sums of the function
    \(g\) was
    studied by Katz, see~\cite[Theorem~5.4.1]{katz:sommes-exponentielles}.
    In this case, he proved the L-series is a
    polynomial or a reciprocal of a polynomial, whose degree can be calculated
    using Chern classes.
    We could also deduce these results from Theorem~\ref{theorem:main-weak}.

    Katz also proved the Frobenius eigenvalues are pure.
    Our method is not capable of proving this purity result.

  \item[(b)] \textit{p-adic theorems.}
    When \(k\) is a finite field and
    \(X = \mathbb{G}_{\mathrm{m}}^{n} \times \mathbb{G}_{\mathrm{a}}^{m}\),
    the \(p\)-adic properties of the L-series were studied by
    E.~Bombieri~\cite{bombieri:exponential-sums-finite-fields},
    and later greatly expanded
    by A.~Adolphson and S.~Sperber~\cite{adolphson-sperber:newton-polyhedra-degree-l-function}.
    The studies of Adolphson and Sperber are based on Dwork's
    works~\cite{dwork:zeta-function-hypersurfaces-1,dwork:zeta-function-hypersurface-2},
    and methods from singularity theory and toric
    geometry.

    The upshot is that Adolphson and Sperber introduced
    a complex of \(p\)-adic Banach spaces, and an operator \(\alpha\)
    with trace acting on the complex, such that the hyper-determinant of \(\alpha\) gives
    rise to the L-series of the exponential sum.
    Moreover, \emph{when the function \(f\) is ``nondegenerate and convenient''},
    Adolphson and Sperber proved that the cohomology spaces of this complex are
    finite dimensional, and concentrated in a single cohomological degree. The
    dimensions of these cohomology spaces are the same as the algebraic de~Rham
    cohomology spaces.
  \end{enumerate}

  Even when the exponential sum \emph{is} defined on
  \(\mathbb{G}_{\mathrm{m}}^n\times \mathbb{G}_{\mathrm{a}}^{m}\),
  Theorem~\ref{theorem:main} could imply results that cannot be deduced from the
  classical theorems, as it could handle Newton-degenerate functions.
  See Example~\ref{example:illustration-newton-degenerate} for a (trivial)
  illustration and Example~\ref{example:b3-arrangement} for two more complicated
  cases.
\end{situation}

\begin{remark}%
  The Dwork--Bombieri--Adolphson--Sperber complex is rather different from the complex
  computing the rigid cohomology (\ref{sit:intro}/\ref{item:rigid-coh}) in two respects:
  \begin{itemize}
  \item the Dwork--Bombieri--Adolphson--Sperber complex is a complex of Banach
    spaces, whereas the complex computing rigid cohomology is a complex of
    ind-Banach spaces, and is never Banach;
  \item the Dwork--Bombieri--Adolphson--Sperber complex should be thought of as
    a twisted de~Rham complex on a rigid analytic subspace of a toric variety, but the
    rigid cohomology is defined via a complex on the rigid analytic torus (as a dagger space).
  \end{itemize}

  Nevertheless, a comparison theorem between a dagger variant of
  Adolphson--Sperber cohomology and rigid cohomology has been proven by Peigen
  Li~\cite{li_peigen:adolphson-sperber-vs-rigid}.
\end{remark}

\begin{situation}%
  \textbf{About the proof.}
  The strategy is to reduce the problems to \(\mathbb{A}^1\) via taking direct
  images, and then use the theory of \(p\)-adic ordinary differential equations to
  deal with the problems on \(\mathbb{A}^1\). There are two major inputs, namely
  % The
  % first is the \(p\)-adic local monodromy theorem or the \(p\)-adic Turrittin
  % theorem (the Crew conjecture, proposed
  % in~\cite{crew:finiteness-cohomology-overconvergent-isocrystal,tsuzuki:slope-filtration-quasi-unipotent-overconvergent-f-isocrystals},
  % proved independently by
  % André~\cite{andre:hasse-arf-filtration-and-p-adic-monodromy},
  % Mebkhout~\cite{mebkhout:p-adic-turrittin-and-p-adic-monodromy-theorem}, and
  % Kedlaya~\cite{kedlaya:p-adic-monodromy}).
  Christol and Mebkhout's characterization of ``\(p\)-adic regular
  singularity''~\cite{christol-mebkhout:index-theorems-of-p-adic-differential-equations-2},
  and Robba's index computation using radii of
  convergence~\cite{robba:index-p-adic-differential-operators-3}. It should also
  be obvious that many of the arguments we present below are influenced by
  Baldassarri~\cite{baldassarri:rigid-algebraic-comparison-theorem-1}, and
  Chiarellotto~\cite{chiarellotto:rigid-algebraic-comparison-theorem}.

  In Section~\ref{sec:radii} we recall the notion of radius of convergence of a
  differential module. In Section~\ref{sec:indices} we explain how to use
  Robba's index theorem to make local calculations. In Section~\ref{sec:rig} we
  globalize the results of Section~\ref{sec:indices} and prove the main theorem.
  Section~\ref{sec:examples} contains some examples. The last section discusses
  the Higgs variant of Theorem~\ref{theorem:main-weak}.
\end{situation}

\noindent%
\textbf{Acknowledgment.} We are grateful to Daqing Wan for communications on
several examples of exponential sums and for pointing out Katz's
theorem.

\section{Radii of convergence}
\label{sec:radii}
This section reviews the notion of radius of convergence of differential
module. We also recall a few basic results, well-known to experts, that we will be
using later.

\begin{situation}\label{situation:prelim-notation}%
  \textbf{Notation.} In this section we fix the following notation.
  \begin{itemize}[wide]
  \item Let \(K\) be a complete ultrametric field of
    characteristic \(0\). Assume that the residue field of \(K\) is of
    characteristic \(p > 0\). Let \(\pi\) be an element of \(K\) satisfying
    \(\pi^{p-1}+p=0\). The field \(K\) is the ``base field'' where spaces are
    defined.

  \item Let \(\Omega\) be an algebraically closed complete ultrametric field
    containing \(K\), such that \(|\Omega| = \mathbb{R}_{\geq 0}\).
    Assume that the
    residue field \(\Omega\) is a transcendental extension of the residue field
    of \(K\). The field \(\Omega\) plays an auxiliary role which will give the
    so-called ``generic points'' to geometric objects.

  \item Let \(I\) be a connected subset of \(\mathbb{R}_{\geq 0}\).
    Let \(\Delta_{I}\) be the rigid analytic space whose underlying set is
    \begin{equation*}
      \{x \in K^{\mathrm{alg}} : |x| \in I\}.
    \end{equation*}
    Let \(\mathbb{D}^{\pm}(x;r)\) be the rigid analytic space whose underlying set is the
    open/closed disk of radius \(r\) centered at \(x \in K\).
    We use \(\mathcal{O}\) to denote the sheaf of rigid analytic functions on
    these spaces.

  \item In addition to the rigid analytic spaces above, we also consider their
    extensions to \(\Omega\).
    Let \(\Delta_{I,\Omega}\) be the analytic space over \(\Omega\) whose points
    are \(\{x \in \Omega : |x| \in I\}\).
    Similarly, for each \(\xi \in \Omega\) and \(r \in \mathbb{R}_{\geq 0}\),
    define
    \(\mathbb{D}_{\Omega}^{+}(\xi;r) = \{x \in \Omega : |x - \xi| \leq r\}\), and
    \(\mathbb{D}_{\Omega}^{-}(\xi,r) = \{x \in \Omega : |x - \xi| < r\}\).

  \item By a ``differential module'' over \(\Delta_{I}\) or
    \(\mathbb{D}^{\pm}(a;r)\), we shall mean a finite free
    \(\mathcal{O}\)-module \(\mathcal{E}\) over \(\Delta_{I}\) or \(\mathbb{D}^{\pm}(a;r)\)
    equipped with an integrable connection.

  \item  The \(\rho\)-Gauss norm on \(K[x]\) is
    \begin{equation*}
      \textstyle  \left| \sum\limits_{i\in \mathbb{N}} a_i x^i \right|_{\rho} = \sup \{|a_i| \rho^i : i \in \mathbb{N}\}.
    \end{equation*}
    It extends to the field \(K(x)\) of rational functions naturally.
    For \(\rho \in \mathbb{R}_{>0}\), denote by \(F_{\rho}\) the completion \(K(x)\)
    with respect to the \(\rho\)-Gauss norm \(|\cdot|_{\rho}\).
    It turns out that \(F_{\rho}\) is also a complete ultrametric field, and
    carries a continuous extended derivation \(\mathrm{d}/\mathrm{d}x\).
  \item A \emph{differential module} over \(F_{\rho}\) is a finite dimensional
    \(F_{\rho}\)-vector space \(V\) equipped with an \(K\)-linear map
    \(D\colon V \to V\),
    such that for any \(a \in F_{\rho}\), any \(v \in V\), the Leibniz rule
    \(D(av) = \frac{\mathrm{d}a}{\mathrm{d}x}v + a D(v)\) holds. It follows that
    \(D\) is automatically continuous.
  \end{itemize}
\end{situation}

\begin{remark}[From \(\mathcal{O}\)-modules to \(F_{\rho}\)-modules]%
  \label{remark:f-rho-as-residue-field-of-generic-point}
  Let \(I \subset \mathbb{R}_{\geq 0}\) be an interval.
  Let \(\rho > 0\) be an element in \(I\).
  Then there is a natural continuous homomorphism
  \begin{equation*}
    \varphi_{\rho}\colon \mathcal{O}(\Delta_{I}) \to F_{\rho}
  \end{equation*}
  such that \(|\varphi(f)|_{\rho} = |f|_{\rho}\).

  To construct the homomorphism, first assume \(I=[a,b]\) is a closed interval.
  Then each rigid analytic function \(f\) on \(\Delta_{I}\) could be written as
  \begin{equation*}
    f(x) = \sum_{n\in \mathbb{Z}} c_{n} x^{n}, \quad a_n \in K,
  \end{equation*}
  such that \(c_n b^{n} \to 0\) as \(n \to +\infty\),
  and \(c_n a^{n} \to 0\) as \(n \to -\infty\).
  In particular, \(c_n \rho^n \to 0\) as \(n \to \pm \infty\).
  Set \(P_N(x) = \sum_{|n| \leq N} c_n x^{n} \in K(x)\).
  Then \(P_{N} \to f\) with respect to the supremum norm
  of \(\mathcal{O}(\Delta_{I})\).
  The condition that \(c_n \rho^n \to 0\) implies that
  \((P_{N})_{N=1}^{\infty}\)  is a Cauchy sequence with respect to the
  \(\rho\)-Gauss norm on \(K(x)\).
  Hence \(\lim_{N\to +\infty} P_{N}\)
  exists in \(F_{\rho}\). We define this element to be
  \(\varphi_{\rho}(f)\).

  In general, we choose an interval \([a,b] \subset I\) containing \(\rho\)
  and define \(\varphi_{\rho}(f) = \varphi_{\rho}(f|_{\Delta_{[a,b]}})\).
  One checks that this definition satisfies the required properties.

  Thus, if \(N\) is a free \(\mathcal{O}\)-module on \(\Delta_{I}\) for some
  connected \(I \subset \mathbb{R}_{\geq 0}\), \(\rho \in I\), then the pullback
  \(V=N \otimes_{\mathcal{O}(\Delta_{I}),\varphi_{\rho}} F_{\rho}\) gives rise to a differential
  module over \(F_{\rho}\), with \(D(n)=\nabla_{\frac{\mathrm{d}}{\mathrm{d}x}}n\)
  for any \(n \in N\). For simplicity we shall write this tensor
  product simply by \(V=N \otimes_{\mathcal{O}}F_{\rho}\).
\end{remark}

\begin{definition}%
  \label{definition:radius}
  Let \(V\) be a vector space over \(F_{\rho}\) equipped with a norm \(|\cdot|\).
  Recall that the \emph{operator norm} of an operator \(T\) on \(V\) is defined
  to be \(|T|_{V} = \sup_{v\in V\setminus \{0\}} |T(v)|/|v|\); and the
  \emph{spectral radius} of \(T\) to be the quantity
  \begin{equation*}
    |T|_{\mathrm{sp},V} = \lim_{s\to \infty} |T^{s}|_{V}^{1/s}.
  \end{equation*}
  The operator norm of \(T\) certainly depends upon the norm, but
  two equivalent norms determine the same spectral
  radius~\cite[Proposition~6.1.5]{kedlaya:p-adic-differential-equations}.

  Let \(V\) be a differential module over \(F_{\rho}\).
  Then the \emph{radius of convergence} of \(V\) is
  \begin{equation*}
    \mathrm{R}(V) = |\pi| \cdot |D|_{\mathrm{sp},V}^{-1},
  \end{equation*}
  % Following~\cite[Definition~9.4.7]{kedlaya:p-adic-differential-equations} we
  % call the the quantity
  % \begin{equation*}
  %   \mathrm{IR}(V) = \rho^{-1} \mathrm{R}(V).
  % \end{equation*}
  % the \emph{intrinsic radius of convergence} of \(V\).

  % By~\cite[Lemma~9.4.6(c)]{kedlaya:p-adic-differential-equations}, if \(V_1\)
  % and \(V_2\) are two differential modules over \(F_{\rho}\), one has
  % \(\mathrm{R}(V_1 \otimes V_2) \geq \min\{\mathrm{R}(V_1), \mathrm{R}(V_2)\}\).
  % Moreover, the equality holds when \(\mathrm{R}(V_1) \neq \mathrm{R}(V_2)\).
  For $0<r<1$, we say a differential module \(M\) over
  \(\Delta_{\interval[open right]{r}{1}}\) or \(\Delta_{\interval[open]{r}{1}}\)
  is \emph{overconvergent} (or \emph{solvable} at \(1\)) if
  \(\lim_{\rho\to 1^{-}}\rho^{-1}\mathrm{R}(M\otimes F_{\rho}) = 1\).
\end{definition}

\begin{example}\label{example:trivial-diff-module-radius}%
  The spectral radius of the trivial differential module
  \((F_{\rho},\mathrm{d}/\mathrm{d}x)\) equals \(|\pi|\rho^{-1}\).
  Thus its radius of convergence equals \(\rho\).
  As the spectral radius of a differential module \(V\) is bigger than or equal to
  that of
  \(\mathrm{d}/\mathrm{d}x\)~\cite[Lemma~6.2.4]{kedlaya:p-adic-differential-equations},
  we know that the radius of convergence of any differential module
  over \(F_{\rho}\) is in the range \(\interval[open left]{0}{\rho}\).
\end{example}

The terminology ``radius of convergence'' comes from the so-called
``Dwork transfer theorem'', which we record below.

\begin{theorem}[Dwork]%
  \label{theorem:transfer}
  Let \(M\) be a differential module over \(\Delta_{I}\) of rank \(n\).
  Let \(\rho \in I\). Then the following two conditions are equivalent.
  \begin{enumerate}
  \item The radius of convergence of \(M \otimes_{\mathcal{O}}F_{\rho}\) is \(R\).
  \item For any \(\xi \in \Delta_{\{\rho\},\Omega}\),
    the restriction of \(M\) to the open disk
    \(\mathbb{D}_{\Omega}^{-}(\xi;R)\) has \(n\) linearly independent horizontal
    sections.
  \end{enumerate}
\end{theorem}

\begin{proof}
  The proof of (1) \(\Rightarrow\) (2) is
  \cite[Theorem~9.6.1]{kedlaya:p-adic-differential-equations}.
  Here the variable \(t\) used by Kedlaya  is \(t-\xi\) in our context, and the
  differential module considered by Kedlaya is the restriction of \(M\) to the
  open disk (thus the connection matrix automatically has entries in the ring of
  analytic elements, fulfilling the hypothesis of the cited theorem).

  The proof of (2) \(\Rightarrow\) (1) is
  \cite[Proposition~9.7.5]{kedlaya:p-adic-differential-equations}.
  Here it is important to note that we should consider the field \(\Omega\)
  instead of \(K\) itself,  so that we have enough ``generic points'' available.
\end{proof}

The ``most'' convergent differential module over \(\Delta_{I}\) are those
who satisfy the Robba condition.

\begin{definition}%
  \label{definition:robba-condition}
  Let \(M\) be a differential module over \(\Delta_{I}\).
  \(M\) is said to satisfy the
  \emph{Robba condition}
  if for any \(\rho \in I\),
  the differential module \(M\otimes_{\mathcal{O}}F_{\rho}\) has radius of
  convergence equal to \(\rho\).
\end{definition}

\begin{lemma}%
  \label{lemma:robba-subquotient}
  Let \(M\) be a differential module over \(\Delta_{I}\) satisfying the Robba
  condition.
  Then any subquotient differential module of \(M\) satisfies the Robba condition.
\end{lemma}

\begin{proof}
  Let \(x_0\) be a point of \(\Delta_{I,\Omega}\).

  Let \(M''\) be a quotient of
  \(M\). Then the horizontal basis of
  \(M|_{{\mathbb{D}_{\Omega}^{-}(x_0;|x_0|)}}\) is sent to a set of horizontal sections
  of \(M''\) which necessarily generate \(M''|_{{\mathbb{D}_{\Omega}^{-}(x_0;|x_0|)}}\).
  This implies that \(M''|_{{\mathbb{D}_{\Omega}^{-}(x_0;|x_0|)}}\) is trivial.

  Let \(M'\) be a differential submodule of \(M\). Let \(M''=M/M'\).
  Since
  \begin{equation*}
    \mathrm{H}^{0}_{\mathrm{DR}}(\mathbb{D}_{\Omega}^{-}(x_0;|x_0|),M) \to
    \mathrm{H}^{0}_{\mathrm{DR}}(\mathbb{D}_{\Omega}^{-}(x_0;|x_0|),M'')
  \end{equation*}
  is surjective, and since both \(M\) and \(M''\) are trivial differential
  modules over \(\mathbb{D}_{\Omega}^{-}(x_0;|x_0|)\), the dimension of horizontal
  sections of \(M'\) over \(\mathbb{D}_{\Omega}^{-}(x_0;|x_0|)\) equals the rank of
  \(M'\), i.e., \(M'\) is trivial on \(\mathbb{D}_{\Omega}^{-}(x_0;|x_0|)\).
\end{proof}

Finally, we quote a theorem due to Christol and
Mebkhout~\cite{christol-mebkhout:index-theorems-of-p-adic-differential-equations-2}.
See also~\cite{dwork:exponents-p-adic-differential-modules}
and~\cite[Theorem~13.7.1]{kedlaya:p-adic-differential-equations}.

\begin{theorem}[Christol--Mebkhout]%
  \label{example:regular-singularity}
  Let \(M\) be a differential module over \(\Delta_{\interval[open]{0}{1}}\).
  Assume that there exists a basis \(e_1,\ldots,e_n\) of \(M\) such that
  \begin{itemize}
  \item the entries of the matrix representation \(\eta\) of
    \(\nabla_{t\frac{\mathrm{d}}{\mathrm{d}t}}\) with respect to this basis
    belong to \(\mathcal{O}(\mathbb{D}^{-}(0;1))\),
  \item \(M\) is overconvergent (see Definition~\ref{definition:radius}),
    and
  \item the eigenvalues of \(\eta(0)\) belong to \(\mathbb{Z}_{p} \cap \mathbb{Q}\).
  \end{itemize}
  Then there exists a basis of \(M\) under which the matrix of
  \(\nabla_{t\frac{\mathrm{d}}{\mathrm{d}t}}\) has entries in
  \(\mathbb{Z}_{p}\). Moreover, \(M\) satisfies the Robba condition.
\end{theorem}

\section{Indices of differential modules}
\label{sec:indices}

In this section, we use the notion
of radii of convergence and Robba's index theorem to
prove some cohomology groups are zero. The
notation and conventions made in Paragraph~\ref{situation:prelim-notation} are
still enforced in this section.

\begin{lemma}%
  \label{lemma:no-section-if-radius-small}
  Let \(N\) be a differential module over \(\Delta_{[a,b]}\).
  Assume that there exists \(a \leq \rho \leq b\) such that the
  radius of convergence of any differential submodule of \(N\otimes F_{\rho}\)
  is \(<\rho\).
  Then \(\mathrm{H}_{\mathrm{DR}}^0(\Delta_{[a,b]},N) = \{0\}\).
\end{lemma}

\begin{proof}
  Let \(N'\) be the \(\mathcal{O}\)-submodule of \(N\) generated by horizontal
  sections of \(N\). Since \(\mathcal{O}(\Delta_{[a,b]})\) is noetherian, \(N'\)
  is finitely generated, and is equipped with a trivial connection. It follows that
  \(N'\) is a finite free differential module over \(\Delta_{[a,b]}\), say of
  rank \(r\). In particular, it is flat over \(\mathcal{O}\). It follows that
  \(N'\otimes_{\mathcal{O}} F_{\rho}\) is a trivial differential submodule of
  \(N\otimes_{\mathcal{O}} F_{\rho}\) of rank \(r\). Being trivial,
  \(N'\otimes_{\mathcal{O}}F_{\rho}\) has radius of convergence equal to
  \(\rho\). The hypothesis then implies that \(r = 0\), in other words,
  \(N' = 0\) and \(\mathrm{H}^{0}_{\mathrm{DR}}(\Delta_{[a,b]},N)=\{0\}\).
\end{proof}

We give a simple calculation of the radius of convergence.

\begin{example}%
  \label{example:radius-of-exponential}
  Let \(L\) be the differential module on
  \(\Delta_{\interval[open]{0}{1}}=\mathbb{D}^{-}(0;1)\setminus \{0\}\)
  defined by the system
  \begin{equation*}
    \frac{\mathrm{d}}{\mathrm{d}x} - \frac{\pi}{x^2}.
  \end{equation*}
  Then for any \(0 < \rho < 1\), the radii of convergence of both
  \(L\) and its dual are equal to \(\rho^2 <\rho\).
\end{example}

\begin{proof}[Proof \textup{(cf.~{\cite[5.4.2]{robba:index-p-adic-differential-operators-3}})}]
  Let \(t \in \Omega\) be a any point of radius \(\rho\).
  Let \(x = t + y\).
  A horizontal section of the differential system is given by
  \(\exp\left(-\pi \left( \frac{1}{t+y} - \frac{1}{t} \right) \right)\).
  The Taylor series for \(\frac{1}{t+y} - \frac{1}{t}\) at \(y = 0\) is
  \begin{equation}\label{eq:exp-radius}
    \sum_{\nu=1}^{\infty} \pm \frac{1}{t^{\nu+1}} y^{\nu}.
  \end{equation}
  For each \(r < \rho\), the \(r\)-Gauss norm of this Taylor series equals
  \begin{equation*}
    \sup \{r^{\nu}/\rho^{\nu+1} : \nu \in \mathbb{N}\} = r/\rho^{2}.
  \end{equation*}
  Thus \(\exp\left( \pi \left( \frac{1}{t+y} - \frac{1}{t} \right) \right)\)
  converges for \(y\) in the open disk \(\mathbb{D}^{-}(0;r)\) where \(r < \rho^2\).

  We claim that
  \(u(y)=\exp\left( \pi \left( \frac{1}{t+y} - \frac{1}{t} \right) \right)\)
  diverges for some \(y\) satisfying \(|y|=\rho^{2}\).

  Indeed, write \(\frac{1}{t+y}-\frac{1}{t}\)
  as \(\frac{1}{t^2}y + h(y)\). Then \(|h(y)|_{\rho^{2}} < 1\) and
  \(\exp\left(-\pi h(y)\right)\) is convergent for all \(y \in \Omega\)
  such that \(|y|=\rho^{2}\). It follows that
  \begin{equation*}
    \exp\left( \frac{\pi y}{t^{2}}  \right) = u(y) \cdot \exp\left(-\pi h(y)\right)
  \end{equation*}
  is convergent on \(\{y \in \Omega : |y| = \rho^{2}\}\), if \(u(y)\) were
  convergent there. This is absurd, as
  \(\sum \pi^{n}/n!\) diverges. Thus, the radius of convergence of \(L\) equals
  \(\rho^2 < \rho\).

  The dual of \(L\) is the differential module associated with the differential
  system
  \begin{equation*}
    \frac{\mathrm{d}}{\mathrm{d}x} + \frac{\pi}{x^{2}},
  \end{equation*}
  and the argument is identical.
\end{proof}

\begin{lemma}%
  \label{lemma:no-section-after-twisting}
  Suppose \([a,b]\) is an interval contained in \(\interval[open]{0}{1}\).
  Let \(M\) be a differential module on \(\Delta_{[a,b]}\) satisfying the Robba
  condition.
  Then we have
  \[
    \mathrm{H}_{\mathrm{DR}}^0(\Delta_{[a,b]},M \otimes L) = 0.
  \]
\end{lemma}

\begin{proof}
  By Lemma~\ref{lemma:robba-subquotient}, any submodule of \(M\) satisfies the
  Robba condition. By
  \cite[Lemma~9.4.6(c)]{kedlaya:p-adic-differential-equations} and
  Example~\ref{example:radius-of-exponential},
  the radii of convergence of all differential submodules of \(M \otimes L\)
  are equal to the radius of convergence of \(L\),
  which is \(<\rho\) at \(F_{\rho}\).
  Thus Lemma~\ref{lemma:no-section-if-radius-small} implies the desired result.
\end{proof}

The above vanishing of cohomology groups implies the vanishing of the cohomology
groups of some special differential modules over the Robba ring.

\begin{definition}%
  \label{definition:robba-ring}
  The \emph{Robba ring} is the colimit
  \begin{equation*}
    \mathcal{R} = \mathop{\mathrm{colim}}\limits_{r\to 1^-} \mathcal{O}(\Delta_{\interval[open]{r}{1}})
    = \mathop{\mathrm{colim}}\limits_{r\to 1^-} \mathcal{O}(\Delta_{\interval[open right]{r}{1}}).
  \end{equation*}
  It is equipped with a derivation \(\mathrm{d}/\mathrm{d}x\).
  As in Paragraph~\ref{situation:prelim-notation}, one can define the notion of a
  differential module over \(\mathcal{R}\).
  Suppose \(M\) is a differential module over \(\mathcal{R}\) with derivation
  \(D\). Define \(\mathrm{H}^0_{\mathrm{DR}}(\mathcal{R},M) = \operatorname{Ker}D\), and
  \(\mathrm{H}^1_{\mathrm{DR}}(\mathcal{R},M) = \operatorname{Coker}D\).
\end{definition}

\begin{lemma}%
  \label{lemma:no-section-and-h1-over-robba}
  Let \(M\) be a differential module on the space
  \(\mathbb{D}^{-}(0;1) \setminus \{0\}\).
  Assume that \(M\) satisfies the hypothesis of
  Theorem~\ref{example:regular-singularity}.
  Let \(L\) be as in Example~\ref{example:radius-of-exponential}.
  Then we have
  \begin{equation*}
    \mathrm{H}^0_{\mathrm{DR}}(\mathcal{R},(M \otimes L) \otimes \mathcal{R})
    = \mathrm{H}^{1}_{\mathrm{DR}}(\mathcal{R},(M\otimes L)\otimes \mathcal{R}) = 0.
  \end{equation*}
\end{lemma}

\begin{proof}
  % Since the radii of convergence of both \(M\) and \(L\)  over \(F_{\rho}\)
  % are convergent to \(1\) as \(\rho \to 1\),
  % both \(M \otimes \mathcal{R}\) and \(L \otimes \mathcal{R}\) are
  % overconvergent.
  % By the solution of Crew's conjecture~\cite{andre:hasse-arf-filtration-and-p-adic-monodromy,mebkhout:p-adic-turrittin-and-p-adic-monodromy-theorem,kedlaya:p-adic-monodromy}),
  % and \cite[Corollary~6.8]{crew:finiteness-cohomology-overconvergent-isocrystal},
  % we see both \(\mathrm{H}_{\mathrm{DR}}^0(\mathcal{R},M \otimes L)\) and
  % \(\mathrm{H}^{1}_{\mathrm{DR}}(\mathcal{R},M\otimes L)\) are finite dimensional, and
  % \(\mathrm{H}_{\mathrm{DR}}^{1}(\mathcal{R},M\otimes L) \cong \mathrm{H}_{\mathrm{DR}}^0(\mathcal{R},M^{\vee} \otimes L^{\vee})^{\vee}\).
  Let \(s\) be a horizontal section of \(M \otimes L\) over the Robba ring,
  then there must exist \(r < 1\) such that the section is defined on the
  annulus \(\Delta_{\interval[open]{r}{1}}\), and thus on the annulus
  \(\Delta_{[a,b]}\) for any \([a,b] \subset \interval[open]{r}{1}\).
  By Lemma~\ref{lemma:no-section-after-twisting},
  we know the section has to be zero on all such \(\Delta_{[a,b]}\),
  and hence the section must be globally zero. This implies that
  \(\mathrm{H}_{\mathrm{DR}}^0(\mathcal{R},(M\otimes L)\otimes \mathcal{R})=\{0\}\).
  % and by duality and the vanishing of
  % sections of \(M^{\vee} \otimes L^{\vee}\), we conclude that
  % \(\mathrm{H}_{\mathrm{DR}}^1(\mathcal{R},M\otimes L) = 0\).

  % If \(M\) is defined over some finite extension of \(W(k)[1/p]\) equipped with
  % a lift of Frobenius, and \(M\) admits
  % a Frobenius structure, the lemma then follows from
  % the solution of Crew's conjecture~\cite{andre:hasse-arf-filtration-and-p-adic-monodromy,mebkhout:p-adic-turrittin-and-p-adic-monodromy-theorem,kedlaya:p-adic-monodromy},
  % \cite[Corollary~6.8]{crew:finiteness-cohomology-overconvergent-isocrystal},
  % and a duality argument.
  % In the present context, we can argue as follows.
  Thanks to the vanishing of \(\mathrm{H}^0\), the vanishing of
  \(\mathrm{H}^{1}\) will follow if we can show the Euler characteristic of
  \(M\otimes L \otimes \mathcal{R}\) is zero.
  By Theorem~\ref{example:regular-singularity}, we can make a
  \(\mathbb{Q}_p\)-linear change of bases to put the matrix of
  \(\nabla_{x\frac{\mathrm{d}}{\mathrm{d}x}}\) in a upper triangular form.
  Thus there is a filtration
  \begin{equation*}
    M = M_n \supset M_{n-1} \supset \cdots \supset M_1  \supset M_{-1} = \{0\}
  \end{equation*}
  of \(M\) by differential submodules such that the subquotients
  \(M_{i}/M_{i-1}\) are of rank \(1\), necessarily satisfy the Robba condition
  (Lemma~\ref{lemma:robba-subquotient}).
  Thus, the vanishing of the Euler characteristic of
  \(M\otimes L\otimes \mathcal{R}\) is implied by the vanishing of the Euler
  characteristic of \(M_{i}/M_{i-1} \otimes L\). Thus we may assume \(M\) has
  rank \(1\) and is defined by a differential equation
  \(\frac{\mathrm{d}}{\mathrm{d}x}-\frac{c}{x}\), with \(c \in \mathbb{Z}_{p}\cap \mathbb{Q}\).

  The de~Rham complex for \(M \otimes L \otimes \mathcal{R}\) is the filtered
  colimit
  \begin{equation*}
    \mathop{\operatorname{colim}}_{r\to 1^{-}} \mathrm{DR}(M\otimes L\otimes \mathcal{O}_{\Delta_{\interval[open right]{r}{1}}}).
  \end{equation*}
  Choose a sequence of numbers \(a_n(r)\) such that \(a_n(r)\uparrow 1^{-}\).
  The above complex reads
  \begin{equation*}
    \mathop{\operatorname{colim}}_{r\to 1^{-}}\lim_{n} \mathrm{DR}(M\otimes L|_{\Delta_{[r,a_n(r)]}}).
  \end{equation*}
  The transition maps in the inverse system having dense images, one knows that
  \(R^{1}\lim_n\) equals zero (Kiehl's Theorem B).
  Thus it suffices to prove the Euler characteristic of
  \begin{equation}\label{eq:robba-characteristic}
    \nabla_{\frac{\mathrm{d}}{\mathrm{d}x}} \colon
    M \otimes L|_{\Delta_{[r,a_n(r)]}} \to  M \otimes L|_{\Delta_{[r,a_n(r)]}}
  \end{equation}
  is zero.

  Note that \(M \otimes L\) is defined by a differential operator of order one with
  coefficients in \(\Omega(x)\). Thus we can use Robba's index
  theorem~\cite[Proposition~4.11]{robba:index-p-adic-differential-operators-3},
  which in our situation asserts that, if \(I = [a,b]\), then
  \begin{equation}\label{eq:robba-index-theorem}
    \chi(\nabla_{\mathrm{d}/\mathrm{d}x}) =
    \frac{\mathrm{d}\log \mathrm{R}((M \otimes L) \otimes_{\mathcal{O}}F_{\rho})}{\mathrm{d}\log \rho}\Big|_{\rho=a} -
    \frac{\mathrm{d}\log \mathrm{R}((M\otimes L)\otimes_{\mathcal{O}}F_{\rho})}{\mathrm{d}\log \rho}\Big|_{\rho=b}.
  \end{equation}
  Since the radius of convergence of \(M \otimes L\) at \(F_{\rho}\) always equals
  \(\rho^{2}\),
  both quantities of the right hand side of the displayed equation are equal to
  \(2\). Thus the Euler characteristic is zero. This completes the proof of the
  lemma.
\end{proof}

\begin{lemma}%
  \label{lemma:acyclicity-open-disk}
  Let \(M\) be differential module over \(\Delta_{\interval[open]{0}{1}}\) satisfying
  the hypothesis of Theorem~\ref{example:regular-singularity}.
  Let \(L\) be as in Example~\ref{example:radius-of-exponential}. Then
  \begin{align*}
    \mathrm{H}^{0}_{\mathrm{DR}}(\Delta_{\interval[open]{0}{1}},M\otimes L) =
    \mathrm{H}^{1}_{\mathrm{DR}}(\Delta_{\interval[open]{0}{1}},M\otimes L)
    = \{0\}.
  \end{align*}
\end{lemma}

\begin{proof}
  The de~Rham cohomology groups are computed by the
  inverse limit
  \begin{equation*}
    \lim_n\left\{
      \Gamma(\Delta_{I_n}, M\otimes L)
      \xrightarrow{\nabla_{\frac{\mathrm{d}}{\mathrm{d}x}}}
      \Gamma(\Delta_{I_n}, M\otimes L)
    \right\}
  \end{equation*}
  % and
  % \begin{equation*}
  %   \lim_n\left\{
  %     \Gamma(\Delta_{I_n}, M\otimes L^{\vee})
  %     \xrightarrow{\nabla_{\mathrm{d}/\mathrm{d}t}}
  %     \Gamma(\Delta_{I_n}, M\otimes L^{\vee})
  %   \right\},
  % \end{equation*}
  where \(I_n\)  is a sequence of closed intervals \([a_n,b_n]\) contained in
  \(\interval[open]{0}{1}\) such that \(a_n \downarrow 0\), \(b_n \uparrow 1\).
  Again, \(R^1\lim\) is zero since the
  transition maps have dense image (Kiehl's Theorem B).
  Thus, it suffices to prove the vanishing of
  \(\mathrm{H}^{\ast}_{\mathrm{DR}}\) for each \(\Delta_{I}\). % Below we only
  % deal with \(L\), as the argument for \(L^{\vee}\)  is identical.

  For each closed interval \(I \subset \interval[open]{0}{1}\), the vanishing of
  zeroth cohomology follows from Lemma~\ref{lemma:no-section-after-twisting}. To
  show that the first cohomology groups are zero, it suffices to prove the Euler
  characteristic of
  \begin{equation*}
    \Gamma(\Delta_I, M\otimes L)
    \xrightarrow{\nabla_{\frac{\mathrm{d}}{\mathrm{d}x}}}
    \Gamma(\Delta_I, M\otimes L)
  \end{equation*}
  is zero.
  By Theorem~\ref{example:regular-singularity}, there is a filtration
  \begin{equation*}
    M = M_n \supset M_{n-1} \supset \cdots \supset M_1  \supset M_{-1} = \{0\}
  \end{equation*}
  of \(M\) by differential submodules such that the subquotients
  \(M_{i}/M_{i-1}\) are of rank \(1\), necessarily satisfy the Robba condition.
  By induction, it suffices to prove the assertion assuming \(M\) has
  rank \(1\). In the rank \(1\) case, Robba's index
  theorem~\eqref{eq:robba-index-theorem}
  applies. Arguing as in the proof of
  Lemma~\ref{lemma:no-section-and-h1-over-robba} shows that the Euler
  characteristic is zero.
\end{proof}

\section{Rigid cohomology associated with a regular function}
\label{sec:rig}
\begin{situation}\label{situation:overall-notation}%
  In this section, we continue using the notation made in Section~\ref{sec:radii}. Thus
  \(k\) is a perfect field of characteristic \(p > 0\),
  \(K\) is a complete discrete valued field of characteristic \(0\) containing
  an element \(\pi\) satisfying \(\pi^{p-1}+p=0\), and the residue field of
  \(K\) is \(k\).
  The ring of integers of \(K\) is denoted by \(\mathcal{O}_K\)

  Our policy is to use Gothic letters to denote schemes over \(\mathcal{O}_K\).
  For a \(\mathcal{O}_K\)-scheme \(\mathfrak{S}\),
  let \(S_0 = \mathfrak{S} \otimes_{\mathcal{O}_K} k\),
  \(S = \mathfrak{S} \otimes_{\mathcal{O}_K} K\).
  For a finite type \(K\)-scheme \(T\), let \(T^{\mathrm{an}}\) be the rigid analytic space
  associated with \(T\).

  We denote by \(\mathbb{D}^{-}(\infty;r)\) the rigid analytic space
  \(\mathbb{P}^{1}_K\setminus \mathbb{D}^{+}(0;r^{-1})\).
\end{situation}

\begin{situation}\label{situation:proper-case}%
  Let conventions be as in Paragraph~\ref{situation:overall-notation}.
  Let \(f\colon \overline{\mathfrak{X}} \to \mathbb{P}^{1}_{\mathcal{O}_K}\) be a
  proper morphism between smooth \(\mathcal{O}_K\)-schemes.
  We make the following assumption.
  \begin{itemize}
  \item[\((\ast)\)] Let \(S \subset \mathbb{P}^{1}_K\) be the non-smooth locus of
    \(f\colon \overline{X} \to \mathbb{P}^{1}_K\). Then the intersection of
    \(\mathbb{P}^{1}_k\) with the nonsmooth locus of
    \(f\colon \overline{\mathfrak{X}}\to \mathbb{P}^{1}_{\mathcal{O}_K}\)
    is contained in the Zariski closure of \(S\).
    In addition, we assume \(S \cap \mathbb{D}^{-}(\infty;1)\) is a subset of
    \(\{\infty\}\).
  \end{itemize}
  The condition \((\ast)\) will be used to ensure the Gauss--Manin connection on
  \(D^{-}(\infty;1)\) is overconvergent.

  Suppose that \(\overline{\mathfrak{X}}\) has
  relative pure dimension \(n\) over \(\mathcal{O}_K\).
  Let \(\mathfrak{X} = f^{-1}(\mathbb{A}^{1}_{\mathcal{O}_K})\).
  Assume that the polar divisor \(\mathfrak{P} = f^{\ast}(\infty_{\mathcal{O}_K})\) is a
  relative Cartier divisor with relative strict normal crossings over \(\mathcal{O}_{K}\).
  Write \(\mathfrak{P} = \sum m_i \mathfrak{D}_i\),
  where \(\mathfrak{D}_i\) are smooth proper \(\mathcal{O}_K\)-schemes.
  We assume \(p\nmid m_i\) for any \(i\).
  The morphisms \(\overline{X}_0 \to \mathbb{P}^{1}_k\)  and
  \(\overline{X} \to \mathbb{P}^{1}_{K}\) induced by \(f\)
  are still denoted by \(f\).
  This abuse of notation is unlikely to cause confusions.
\end{situation}

\begin{situation}\label{situation:tubes}%
  In order to calculate rigid cohomology, we need to set up some notation for
  tubular neighborhoods.
  For \(r<1\), set \([P_0]_{\overline{\mathfrak{X}},r} = f^{-1}(\mathbb{D}^{+}(\infty;r))\)
  (``closed tubular neighborhood of radius \(r\)''),
  \(\tube{X_0}_{\overline{\mathfrak{X}}}=\overline{X}^{\mathrm{an}}\setminus\bigcup_{r<1}[P_0]_{\overline{\mathfrak{X}},r}\),
  and \(V_r := \overline{X} \setminus [P_0]_{\overline{\mathfrak{X}},r}\).

  Denote by \(j\) the inclusion map
  \(\tube{X_0}_{\overline{\mathfrak{X}}}\to\overline{X}\), and by \(j_r\) the inclusion map
  \( V_r\to \overline{X}\).
\end{situation}

\begin{situation}[The Dwork isocrystal]\label{situation:asd}%
  In this paragraph we explain what the Dwork isocrystal is.
  The affine line \(\mathbb{A}_k^{1}\) sits in the frame
  \((\mathbb{A}^{1}_k \subset \mathbb{P}^{1}_k \subset \widehat{\mathbb{P}}^{1}_{\mathcal{O}_K})\)
  where \(\widehat{\mathbb{P}}^1_{\mathcal{O}_K}\) is the formal completion of the projective
  line over \(\mathcal{O}_K\) with respect to the maximal ideal of \(\mathcal{O}_{K}\). Therefore to
  describe this crystal we only need to write down a presentation of it (as an
  integrable connection) on \(\mathbb{P}_K^{1,\mathrm{an}}\).

  On the rigid analytic projective line, the tubular neighborhood of
  \(\infty_k \in \mathbb{P}^{1}_k\) is the complement of the closed unit disk
  \(\mathbb{D}^{+}(0,1)\) of radius \(1\), i.e., \(\mathbb{D}^{-}(\infty;1)\).
  The analytification of the algebraic integrable connection
  \begin{equation*}
    \nabla_{\text{D}} \colon \mathcal{O}_{\mathbb{A}^{1}_K} \to \Omega^{1}_{\mathbb{A}^{1}_K},\quad
    h \mapsto \mathrm{d}h + \pi h \mathrm{d} t
  \end{equation*}
  is easily seen to have \(\rho\) as its radius of convergence on \(D^{-}(a;\rho)\) for any
  \(|a| \leq 1\), \(\rho < 1\). Its restriction to the tubular neighborhood of \(\infty\)
  is precisely the connection we dealt with in
  Example~\ref{example:radius-of-exponential}, hence has radius of convergence
  equal to \(\rho^{2}\) at \(F_{\rho}\), which converges to \(1\) as
  \(\rho \to 1\). Thus
  \((\mathcal{O}_{\mathbb{A}^{1,\mathrm{an}}_K},\nabla_{\mathrm{D}})\)
  is overconvergent for any open disk. By the theory
  of rigid cohomology
  (\cite[Definition~7.2.14, Proposition~7.2.15]{le-stum:rigid-cohomology}),
  the differential module determined by
  \((\mathcal{O}_{\mathbb{A}^{1,\mathrm{an}}_K},\nabla_{\mathrm{D}})\)
  gives rise to an overconvergent isocrystal on
  \(\mathbb{A}^{1}_{k}\), which could be taken as a coefficient system for the
  rigid cohomology. We denote it by \(\mathcal{L}_{\pi}\), and call it the
  Dwork isocrystal~\cite[\S4.2.1]{le-stum:rigid-cohomology}.
\end{situation}

\begin{situation}\label{situation:exp-module-on-analytic-space}%
  Let notation and conventions be as in Paragraph~\ref{situation:proper-case}.
  Consider the morphism \(f\colon X \to \mathbb{A}^{1}_{K}\).
  Then we can define an algebraic integrable connection on the structure sheaf
  \(\mathcal{O}_{X}\)
  \begin{equation*}
    \nabla_{\pi f}\colon \mathcal{O}_{X} \to \Omega_{X/K}^{1},\quad
    \nabla_{\pi f}(h) = \mathrm{d}h + \pi h \mathrm{d}f.
  \end{equation*}
  This integrable connection is the inverse image of the connection
  \((\mathcal{O}_{\mathbb{A}^{1}},\nabla_{\mathrm{D}})\) via the morphism
  \(f\).

  By analytification, we obtain an analytic connection, still denoted by
  \(\nabla_{\pi f}\), on the rigid analytic space \(X^{\mathrm{an}}\).
\end{situation}

\begin{proposition}%
  \label{proposition:proper-comparison}
  Let notation and conventions be as
  in Paragraphs~\ref{situation:proper-case} -- \ref{situation:exp-module-on-analytic-space}.
  Then for each integer \(m\), the natural maps
  \begin{equation*}
    \mathrm{H}^{m}_{\mathrm{DR}}(X,\nabla_{\pi f}) \to
    \mathrm{H}^{m}_{\mathrm{DR}}(X^{\mathrm{an}},\nabla_{\pi f}) \to
    \mathrm{H}^{m}_{\mathrm{rig}}(X_0/K,f^{\ast}\mathcal{L}_{\pi})
  \end{equation*}
  are isomorphisms of \(K\)-vector spaces.
\end{proposition}

\begin{proof}[Proof that the right hand side arrow is an isomorphism]
  Without loss of generality we could assume \(X\) is irreducible.
  When \(f|_X\colon X \to \mathbb{A}^{1}_K\) is a constant morphism, the
  proposition is trivial. In the sequel we shall assume
  \(f|_X \colon X \to \mathbb{A}^{1}_{K}\) is surjective.

  To begin with, let us write down the complex that computes the rigid cohomology.
  Let
  \begin{equation*}
    j^{\dagger}\Omega^{i}_{\overline{X}^{\mathrm{an}}/K}
    = \mathop{\operatorname{colim}}_{r\to 1^{-}} j_{r\ast} j_{r}^{\ast} \Omega^{i}_{\overline{X}^{\mathrm{an}}/K}.
  \end{equation*}
  The analytification of the
  connection \(\nabla_{\pi f}\) (still denoted by \(\nabla_{\pi f}\)) gives rise
  to an integrable connection
  \begin{equation*}
    \nabla^{\dagger} \colon j^{\dagger} \mathcal{O}_{\overline{X}^{\mathrm{an}}} \to j^{\dagger} \Omega_{\overline{X}^{\mathrm{an}}}^{1},
  \end{equation*}
  which extends to a dagger version of the de~Rham complex
  \begin{equation*}
    \mathcal{DR}(\overline{X}^{\mathrm{an}},\nabla^{\dagger}) : \qquad
    j^{\dagger} \mathcal{O}_{\overline{X}^{\mathrm{an}}} \xrightarrow{\nabla^{\dagger}} j^{\dagger} \Omega^1_{\overline{X}^{\mathrm{an}}}
    \xrightarrow{\nabla^{\dagger}} \cdots \xrightarrow{\nabla^{\dagger}} j^{\dagger}\Omega^{n}_{\overline{X}^{\mathrm{an}}}.
  \end{equation*}

  For an admissible open subspace \(V\) of \(\overline{X}^{\mathrm{an}}\),
  let \(\mathcal{DR}(V,\nabla_{\pi f}|_{V})\) be the de~Rham complex
  of \(\nabla_{\pi f}\) restricted to \(V\).
  Set
  \(\mathrm{DR}(V,\nabla_{\pi f}|_{V})=R\Gamma(V,\mathcal{DR}(V,\nabla_{\pi f}|_{V}))\).
  We have
  \[
    \mathrm{DR}(\overline{X}^{\mathrm{an}},\nabla^{\dagger})
    := R\Gamma(\overline{X}^{\mathrm{an}},\mathcal{DR}(X^{\mathrm{an}},\nabla^{\dagger}))
    = \mathop{\mathrm{colim}}\limits_{r\to 1}
    R\Gamma(\overline{X}^{\mathrm{an}},Rj_{r\ast}\mathcal{DR}(V_{r},\nabla_{\pi f}|_{V_r})).
  \]
  A priori, this complex depends upon the formal completion of the scheme
  \(\overline{\mathfrak{X}}\) along the maximal ideal of \(\mathcal{O}_K\). However,
  since \(\overline{\mathfrak{X}}\) is proper and the integrable connection
  \((\mathcal{O}_{X^{\mathrm{an}}},\nabla_{\pi f})\) is overconvergent (being
  the inverse image of an overconvergent integrable connection),
  \cite[Corollary~8.2.3]{le-stum:rigid-cohomology} asserts that it only depends
  upon \(X_0\) and \(f|_{X_0}\colon X_0 \to \mathbb{A}^{1}_k\). We have then
  \begin{equation*}
    \mathrm{DR}(\overline{X}^{\mathrm{an}},\nabla^{\dagger}) = R\Gamma_{\mathrm{rig}}(X_0,f_0^{\ast}\mathcal{L}_{\pi}).
  \end{equation*}

  Next we explain how to compute the cohomology of the de~Rham complex.
  For any \(r\) sufficiently close to \(1\), set
  \(V_r=f^{-1}(\mathbb{D}^{-}(0;r^{-1}))\).
  Then \(X^{\mathrm{an}}\) has an admissible open covering
  \begin{equation*}
    X^{\mathrm{an}} = V_r \cup f^{-1}(\mathbb{D}^{-}(\infty;1)\setminus \{\infty\}).
  \end{equation*}
  In the bounded derived category of \(K\)-vector spaces, the Mayer--Vietoris theorem
  gives an isomorphism between the de~Rham complex
  \(\mathrm{DR}(X,\nabla_{\pi f})\)  and the homotopy fiber of the map
  \begin{equation*}
    \mathrm{DR}(V_r,\nabla_{\pi f})  \oplus
    \mathrm{DR}(f^{-1}(\mathbb{D}^{-}(\infty;1)\setminus\{\infty\}),\nabla_{\pi f}) \to
    \mathrm{DR}(V_r \cap f^{-1}(\mathbb{D}^{-}(\infty;1)), \nabla_{\pi f}).
  \end{equation*}
  Since the colimit, as \(r \to 1^-\), of
  \(\mathrm{DR}(V_r,\nabla_{\pi f})\)
  equals \(R\Gamma_{\mathrm{rig}}(X_0,f_0^{\ast}\mathcal{L}_{\pi})\), in order to
  prove the comparison between rigid and de~Rham cohomology it suffices to prove
  the natural morphism
  \begin{equation}
    \label{eq:comparison}
    \mathrm{DR}(f^{-1}(\mathbb{D}^{-}(\infty;1)\setminus\{\infty\}),\nabla_{\pi f}) \to
    \mathop{\mathrm{colim}}\limits_{r\to 1^-}
    \mathrm{DR}(V_r \cap f^{-1}(\mathbb{D}^{-}(\infty;1)\setminus\{\infty\}), \nabla_{\pi f})
  \end{equation}
  is an isomorphism in the derived category of vector spaces over \(K\).

  Let \(S\) be the finite subspace of \(\mathbb{A}^{1,\mathrm{an}}_K\) containing
  all the critical values of \(f\). Let \(X' = X \setminus f^{-1}(S)\).
  Then the direct image sheaf
  \begin{equation}\label{eq:gm}
    \mathcal{E} = R^{i}f_{\ast}(\Omega_{X'/K}^{\bullet},\mathrm{d}),
  \end{equation}
  is equipped with a Gauss--Manin connection \(\nabla_{\mathrm{GM}}\).
  By projection formula, we have
  \(R^{i}(f|_{X'})_{\ast}(\mathcal{DR}(\mathcal{O}_{X^{\mathrm{an}}},\nabla_{\pi f})|_{X'})\)
  is isomorphic to the analytification of the tensor product
  \begin{equation*}
    \mathcal{E}\otimes (\mathcal{O}_{\mathbb{A}^{1}_K},\nabla_{\mathrm{D}})
  \end{equation*}
  on \(\mathbb{A}^{1,\mathrm{an}}_K\setminus S\).

  % At this point let us quote the following theorem of
  % P.~Berthelot~\cite[Th\'eor\`eme~5]{berthelot:announcement-rigid-cohomology} and
  % N.~Tsuzuki~\cite[Theorem~4.1.1]{tsuzuki:base-change-theorem-and-coherence-in-rigid-cohomology}.

  % \begin{theorem*}%
  %   \label{theorem:tsuzuki}
  %   The vector bundle \(\mathcal{E}^{\mathrm{an}}\) on
  %   \(\mathbb{A}^{1,\mathrm{an}}\setminus S\), together with its Gauss--Manin system,
  %   is overconvergent along \(S \cup \{\infty\}\).
  % \end{theorem*}

  Thus, if
  \(u\colon\mathbb{P}_K^{1,\mathrm{an}}\setminus \tube{S_0}_{\mathbb{P}^1_{\mathcal{O}_K}} \to\mathbb{P}^{1,\mathrm{an}}_K\)
  denotes the open immersion, then
  \begin{equation*}
    u^{\dagger}(\mathcal{E}^{\mathrm{an}},\nabla_{\mathrm{GM}}),
  \end{equation*}
  is the rigid cohomological direct image
  \(R^i (f|_{X'_0})_{\mathrm{rig}\ast}\mathcal{O}_{X'_0}\).

  Using the Leray spectral sequence, we see that in order to prove the morphism
  (\ref{eq:comparison}) is an isomorphism, it suffices to prove that for any
  \(i\), the map
  \begin{equation}
    \label{eq:push-to-curve}
    R\Gamma(\mathbb{D}^{-}(\infty;1)\setminus\{\infty\}, \mathcal{E}^{\mathrm{an}} \otimes
    (\mathcal{O}_{\mathbb{A}^{1,\mathrm{an}}_K},\nabla_{\mathrm{D}}))
    \to    \mathop{\mathrm{colim}}\limits_{r\to 1^{-}}
    R\Gamma(\Delta_{\interval[open left]{1}{r^{-1}}}, \mathcal{E}^{\mathrm{an}} \otimes
    (\mathcal{O}_{\mathbb{A}^{1,\mathrm{an}}_K},\nabla_{\mathrm{D}}))
  \end{equation}
  is an isomorphism.

  We shall show that both sides of~\eqref{eq:push-to-curve} are acyclic.
  % In this statement, the ambient field \(K\) does not play any role any more.
  % We could extend the scalars to an algebraically closed ultrametric field
  % \(\Omega\), and deduce the acyclicity with \(\Omega\) as our base field.
  For convenience, we shall use a coordinate \(x\) around
  \(\infty \in \mathbb{P}^{1,\mathrm{an}}_K\), thus swap \(\infty\) and \(0\).
  Let \(M\) be the restriction of the analytification of
  \(\mathcal{E} = R^i f_{\ast}(\Omega_{X'/K}^{\bullet})\) to the disk
  \(\mathbb{D}^{-}(0;1)\setminus\{0\}\).
  In Lemma~\ref{lemma:direct-image-satisfy-robba-condition}
  below,  we shall explain that \(M\) satisfies the hypothesis of
  Theorem~\ref{example:regular-singularity}. On the other hand, the differential
  module \(L\) considered in Example~\ref{example:radius-of-exponential} is precisely the
  restriction of
  \((\mathcal{O}_{\mathbb{A}^{1}_K},\nabla_{\mathrm{D}})^{\mathrm{an}}\) in the
  vicinity of \(\infty\).

  The right hand side of (\ref{eq:push-to-curve}) now reads
  \begin{equation*}
    \mathop{\operatorname{colim}}_{r\to 1^{-}}
    \left\{
      (M\otimes L)|_{\Delta_{\interval[open right]{r}{1}}}
      \xrightarrow{\nabla_{\frac{\mathrm{d}}{\mathrm{d}x}}}
      (M\otimes L)|_{\Delta_{\interval[open right]{r}{1}}}
    \right\},
  \end{equation*}
  which is precisely the de~Rham complex of \(M \otimes L\) restricted to the
  Robba ring:
  \begin{equation*}
    (M \otimes L)\otimes_{\mathcal{O}} \mathcal{R} \xrightarrow{D} (M \otimes L)\otimes_{\mathcal{O}} \mathcal{R}.
  \end{equation*}
  Thus,
  Lemma~\ref{lemma:no-section-and-h1-over-robba} implies that the right hand side
  of~(\ref{eq:push-to-curve}) is trivial.
  The acyclity of the left hand side of (\ref{eq:push-to-curve})
  follows from Lemma~\ref{lemma:acyclicity-open-disk}. This completes the proof
  that the analytic twisted de~Rham cohomology is isomorphic to the rigid
  cohomology.
\end{proof}

\begin{lemma}%
  \label{lemma:direct-image-satisfy-robba-condition}
  Let notation be as above.
  Then the differential module \(M\) satisfies the hypotheses of
  Theorem~\ref{example:regular-singularity}.
\end{lemma}

\begin{proof}
  By the regularity of the Gauss--Manin system, the differential module
  \(M\) is a restriction of an algebraic integrable connection which is regular
  singular around \(0\). In particular, there exists a basis of \(M\) such that
  the derivation \(\nabla_{\frac{\mathrm{d}}{\mathrm{d}x}}\) is given by
  \begin{equation*}
    \frac{\mathrm{d}}{\mathrm{d}x} - \eta(x)
  \end{equation*}
  where \(\eta\) is a rational function which has at worst a simple pole at
  \(x = 0\) (for example, take the restriction of an algebraic basis and restrict to
  the analytic open \(\mathbb{D}^{-}(\infty;1)\)).

  As we have assumed that the multiplicities of the polar divisor are
  prime to \(p\), the algebraic calculation of exponents of the Gauss--Manin
  system around infinity
  implies that the eigenvalues of \((x\eta)|_{x=0}\) belong to
  \(\mathbb{Z}_{p} \cap \mathbb{Q}\).
  (One can embed field of definition of the variety into \(\mathbb{C}\), then
  use \cite[Expos\'e~XIV,~Proof of Proposition 4.15]{sga7-2} to show that
  the eigenvalue of \(\eta(0)\) with respect to the algebraic
  basis are rational numbers whose denominators are not divisible by \(p\).)

  Finally the overconvergence of the Gauss--Manin system in the lifted situation
  under the hypothesis \((\ast)\) is proved by
  P.~Berthelot~\cite[Th\'eor\`eme~5]{berthelot:announcement-rigid-cohomology} and
  N.~Tsuzuki~\cite[Theorem~4.1.1]{tsuzuki:base-change-theorem-and-coherence-in-rigid-cohomology}.
\end{proof}

\begin{proof}[Proof that the left hand side arrow is an isomorphism] This is a
  theorem of André and
  Baldassarri~\cite[Theorem~6.1]{andre-baldassarri:de-rham-cohomology-differential-modules-on-algebraic-varieties}.
  Since the precise hypothesis of their theorem is not met in the present
  situation (the connection is not defined over a number field), we shall
  nevertheless provide a proof. The key point is a theorem of Clark, which
  states that in a good situation, the analytic index of a differental operator
  equals its formal index.

  By GAGA, the algebraic de~Rham complex is computed by complex
  \begin{equation*}
    \mathcal{DR}_{\mathrm{mer}}\colon \quad
    \mathcal{O}_{\overline{X}^{\mathrm{an}}}(\ast P_{\mathrm{red}}) \xrightarrow{\nabla_{\pi f}}
    \Omega^1_{\overline{X}^{\mathrm{an}}}(\ast P_{\mathrm{red}}) \to \cdots \to
    \Omega^n_{\overline{X}^{\mathrm{an}}}(\ast P_{\mathrm{red}}),
  \end{equation*}
  which is a subcomplex of \(\mathcal{DR}(X^{\mathrm{an}},\nabla_{\pi f})\).
  Here
  \(\Omega^{m}_{\overline{X}^{\mathrm{an}}}({\ast}P_{\mathrm{red}})=\bigcup_{e=1}^{\infty}\Omega^{m}_{\overline{X}^{\mathrm{an}}}(eP_{\mathrm{red}})\).
  Cover \(\overline{X}^{\mathrm{an}}\) by
  \(X^{\mathrm{an}}=f^{-1}(\mathbb{A}^{1,\mathrm{an}})\) and
  \([P_0]_{\overline{\mathfrak{X}},\epsilon}=f^{-1}(\mathbb{D}^{+}(\infty;\epsilon))\),
  the Mayer--Vietoris theorem
  implies that \(\mathrm{DR}(X,\nabla_{\pi f})\) is the homotopy fiber of
  \begin{equation*}
    \mathrm{DR}(X^{\mathrm{an}},\nabla_{\pi f})  \oplus R\Gamma(f^{-1}(\mathbb{D}^{+}(\infty;\epsilon)),\mathcal{DR}_{\mathrm{mer}}) \to
    \mathrm{DR}(X^{\mathrm{an}} \cap f^{-1}(\mathbb{D}^{+}(\infty;\epsilon)), \nabla_{\pi f}).
  \end{equation*}
  Taking colimit with respect to \(\epsilon \to 0\),
  we see it suffices to prove the colimit, as \(\epsilon \to 0\), of the
  following map
  \begin{equation*}
    R\Gamma(f^{-1}(\mathbb{D}^+(\infty;\epsilon)),\mathcal{DR}_{\mathrm{mer}}) \to
    \mathrm{DR}(X^{\mathrm{an}} \cap f^{-1}(\mathbb{D}^{+}(\infty;\epsilon)), \nabla_{\pi f}).
  \end{equation*}
  is a quasi-isomorphism. Again, we shall show both items are acyclic.

  Let \(\mathcal{E}\) be the algebraic Gauss--Manin system on
  \(\mathbb{A}^{1}\setminus S\) as in~\eqref{eq:gm}.
  Let \(\iota\colon \mathbb{A}^{1}_K\setminus S \to \mathbb{P}^{1}_K\) be the inclusion.
  Let \(\mathcal{E}(\ast S) = \iota_{\ast}\mathcal{E}\).
  Using Leray spectral sequence, it suffices to prove
  the de~Rham complex of
  \begin{equation}\label{eq:an-vs-alg-after-colimit-1}
    \mathop{\operatorname{colim}}_{\epsilon\to0}
    (\mathcal{E}^{\mathrm{an}}\otimes(\mathcal{O}_{\mathbb{A}_K^{1,\mathrm{an}}},\nabla_{\mathrm{D}}))
    |_{\mathbb{D}^{+}(\infty;\epsilon)\setminus\{\infty\}}
  \end{equation}
  (``de~Rham complex with essential singularities'')
  and that of
  \begin{equation}\label{eq:an-vs-alg-after-colimit-2}
    \mathop{\operatorname{colim}}_{\epsilon\to0}
    (\mathcal{E}(\ast S)\otimes(\mathcal{O}_{\mathbb{A}^{1}_K}(\ast S),\nabla_{\mathrm{D}}))^{\mathrm{an}}|_{\mathbb{D}^{+}(\infty;\epsilon)}.
  \end{equation}
  (``de~Rham complex with moderate singularities'') are acyclic.

  A similar argument as in the proof of Lemma~\ref{lemma:acyclicity-open-disk}
  using Robba's index theorem yields that
  \eqref{eq:an-vs-alg-after-colimit-1} has zero Euler characteristic and
  vanishing \(\mathrm{H}^0\), thus acyclic.

  Since the de~Rham complex of~\eqref{eq:an-vs-alg-after-colimit-2} is a subcomplex
  of that of~\eqref{eq:an-vs-alg-after-colimit-1},
  its \(\mathrm{H}^{0}\) is also trivial.

  We proceed to prove that~\eqref{eq:an-vs-alg-after-colimit-2}
  has zero Euler characteristic.
  Let \(\mathcal{O}_0\) be the local ring
  \(\mathcal{O}_{\mathbb{P}^{1,\mathrm{an}}_K,\infty}\) with the uniformizer
  \(x\) defined by a coordinate around \(\infty\). Let
  \(\widehat{\mathcal{O}}_{0}\cong K[\![x]\!]\) be its \(x\)-adic completion.
  Then the de~Rham complex of~\eqref{eq:an-vs-alg-after-colimit-2} is that of
  \begin{align*}
    &(\mathcal{E}(\ast S)\otimes(\mathcal{O}_{\mathbb{A}^{1}_K}(\ast S),\nabla_{\mathrm{D}}))^{\mathrm{an}} \otimes_{\mathcal{O}_{\mathbb{P}_K^{1,\mathrm{an}}}} \mathcal{O}_{0} \\
    =&
    (\mathcal{E}(\ast S)\otimes(\mathcal{O}_{\mathbb{A}^{1}_K}(\ast S),\nabla_{\mathrm{D}}))^{\mathrm{an}} \otimes_{\mathcal{O}_{\mathbb{P}_K^{1,\mathrm{an}}}} \mathcal{O}_{0}[1/x].
  \end{align*}
  Choose a cyclic vector
  \cite[II~1.3]{deligne:regular-singularity-differential-equation}
  for the differential module
  \[
    (\mathcal{E}(\ast S)\otimes(\mathcal{O}_{\mathbb{A}^{1}_K}(\ast S),\nabla_{\mathrm{D}}))^{\mathrm{an}}
    \otimes_{\mathcal{O}_{\mathbb{P}_K^{1,\mathrm{an}}}} \mathcal{O}_{0}[1/x]
  \]
  over the differential field \(\mathcal{O}_{0}[1/x]\).
  Thus we obtain a differential operator
  \(u = \sum_{i} a_i\frac{\mathrm{d}^i}{\mathrm{d}x^i}\)
  with \(a_i \in \mathcal{O}_0\), and the Euler characteristic
  equals the index of
  \begin{equation*}
    \mathcal{O}_{0}[1/x] \xrightarrow{u} \mathcal{O}_{0}[1/x].
  \end{equation*}
  On the other hand, Malgrange's index
  theorem~\cite[Th\'eor\`eme~2.1b]{malgrange:on-singular-points-diff-equations}
  implies that the index of
  \begin{equation*}
    \widehat{\mathcal{O}}_{0}[1/x] \xrightarrow{u} \widehat{\mathcal{O}}_0[1/x]
  \end{equation*}
  equals zero. Thus, it suffices to prove the index of
  \begin{equation}\label{eq:clark}
    \widehat{\mathcal{O}}_0/\mathcal{O}_0 \xrightarrow{u}
    \widehat{\mathcal{O}}_0/\mathcal{O}_0
  \end{equation}
  equals zero.
  Since \(\mathcal{E}\) is regular singular, \(\nabla_{\mathrm{D}}\) is
  rank one and irregular, the indicial polynomial of \(u\)
  is zero. Thus the hypothesis on non-Liouville
  difference in
  Clark's theorem~\cite{clark:p-adic-convergence-solutions-linear-des} as stated
  in~\cite[Th\'eor\`eme~2.12]{chiarellotto:rigid-algebraic-comparison-theorem}
  is satisfied, and this theorem implies the index of~\eqref{eq:clark} is zero. This
  completes the proof of Proposition~\ref{proposition:proper-comparison}.
\end{proof}

Next we prove the main result of this note by removing the properness
hypothesis from Proposition~\ref{proposition:proper-comparison}.

\begin{situation}\label{situation:general-case}%
  We follow the conventions made in Paragraph~\ref{situation:overall-notation}.
  Let \(f\colon \overline{\mathfrak{X}} \to \mathbb{P}^{1}_{\mathcal{O}_K}\) be a
  proper morphism between smooth \(\mathcal{O}_K\)-schemes, such that the
  condition~\((\ast)\) in Paragraph~\ref{situation:proper-case} is satisfied.
  Assume \(\overline{\mathfrak{X}}\) has
  relative pure dimension \(n\) over \(\mathcal{O}_{K}\).

  Let \(\mathfrak{X}\) be a Zariski open subscheme of \(\overline{\mathfrak{X}}\),
  such that
  \begin{itemize}
  \item \(\mathfrak{X} \subset f^{-1}(\mathbb{A}^1_{\mathcal{O}_K})\),
  \item \(\mathfrak{D} = \overline{\mathfrak{X}} \setminus \mathfrak{X}\) is a relative
    Cartier divisor with relative strict normal crossings over \(\mathcal{O}_K\).
  \end{itemize}
  The support of the polar divisor \(\mathfrak{P} = f^{\ast}(\infty_{\mathcal{O}_K})\) is
  therefore contained in \(\mathfrak{D}\).
  We could write \(\mathfrak{P} = \sum_{i=1}^{r} m_i \mathfrak{D}_i\).
  Again, we assume \(p\nmid m_i\) for any \(i\).
  Finally we could write
  \(\mathfrak{D}=\sum_{i=1}^{r}\mathfrak{D}_{i}+\sum_{j=1}^{s}\mathfrak{H}_i\).
  Thus, for any subset \(I\) of \(\{1,2,\ldots,r\}\) and any subset \(J\) of
  \(\{1,2,\ldots,s\}\), the intersection
  \begin{equation*}
    \bigcap_{i \in I} \mathfrak{D}_{i} \cap \bigcap_{j \in J} \mathfrak{H}_{j}
  \end{equation*}
  is a smooth proper \(\mathcal{O}_K\)-scheme.
\end{situation}

\begin{theorem}%
  \label{theorem:main}
  Let notation and conventions be as in Paragraph~\ref{situation:general-case}.
  Then for each integer \(m\), the natural maps
  \begin{equation*}
    \mathrm{H}^{m}_{\mathrm{DR}}(X,\nabla_{\pi f}) \to
    \mathrm{H}^{m}_{\mathrm{DR}}(X^{\mathrm{an}},\nabla_{\pi f}) \to
    \mathrm{H}^{m}_{\mathrm{rig}}(X_0/K,f^{\ast}\mathcal{L}_{\pi})
  \end{equation*}
  are isomorphisms.
\end{theorem}

\begin{proof}
  For each subset \(J\) of \(\{1,2,\ldots,s\}\),
  denote by \(H_{J}\) the intersection \(\bigcap_{j\in J}H_j\).
  Then each \(H_{J}\) is a smooth proper \(K\)-scheme, and the restriction of
  \(f\) to \(H_{J}\) is denoted by \(f_{J}\), which is a proper morphism into
  \(\mathbb{P}^{1}_K\). The polar divisor of \(f_{J}\) is the restriction of
  \(P\) to \(H_{J}\). Note that \(\nabla_{\pi f}\) restricts to a connection
  on \(H_{J}\), which is \(\nabla_{\pi f_{J}}\).
  By Proposition~\ref{proposition:proper-comparison},
  the natural maps
  \begin{equation*}
    \mathrm{H}^{m}_{\mathrm{DR}}(H_J,\nabla_{\pi f_J}) \to
    \mathrm{H}^{m}_{\mathrm{DR}}(H_{J}^{\mathrm{an}},\nabla_{\pi f_J}) \to
    \mathrm{H}^{m}_{\mathrm{rig}}(H_{J0}/K,f_I^{\ast}\mathcal{L}_{\pi})
  \end{equation*}
  are isomorphisms.

  By general results in rigid cohomology, there exists a second-quadrant
  spectral sequence with
  \begin{equation*}
    E^{-i,j}_{1} = \bigoplus_{\operatorname{Card}J=i} \mathrm{H}^{j-2i}_{\mathrm{rig}}(H_{J0}/K,f^{\ast}_{I}\mathcal{L}_{\pi})
  \end{equation*}
  which abuts to \(\mathrm{H}^{-i+j}_{\mathrm{rig}}(X_0/K,f^{\ast}\mathcal{L}_{\pi})\).
  (We are unable to locate a precise reference, but this could be deduced from
  \cite[Proposition~7.4.13]{le-stum:rigid-cohomology}.)
  We have similar spectral sequences for the algebraic and analytic de~Rham
  cohomology groups. The \(E_1\)-differentials of these spectral sequence are
  all Gysin maps in the various theories.
  Thus the natural maps between the theories give rise to maps of these
  \(E_1\)-spectral sequences. Proposition~\ref{proposition:proper-comparison}
  implies that these maps are isomorphisms on the
  \(E_1\)-stage. Thus they induce isomorphisms on the abutments.
\end{proof}

\begin{proof}[Proof of Theorem~\ref{theorem:main-weak}]
  Let notation be as in the statement of Theorem~\ref{theorem:main-weak}.
  Let \(\mathbf{K}\) be the field of fractions of \(R\).
  Using resolution of singularities, upon making a finite extension of
  \(\mathbf{K}\) (both rigid and de~Rham cohomology are compatible with
  extensions, so performing a finite extension does not change the result),
  we can embed \(X_{\mathbf{K}}\) into a smooth
  proper \(\mathbf{K}\)-scheme \(\overline{X}_{\mathbf{K}}\) such that
  \(\overline{X}_{\mathbf{K}}\setminus X_{\mathbf{K}}\) has strict normal
  crossings, and \(f\) extends to a morphism
  \(f\colon\overline{X}_{\mathbf{K}} \to \mathbb{P}^{1}_{\mathbf{K}}\).
  There is a Zariski open subset \(U\) of \(\mathrm{Spec}(R)\) over which
  \begin{itemize}
  \item \(\overline{X}_{\mathbf{K}}\) as well as all the intersections of the
    boundary divisors have good reduction at primes in \(U\), and
  \item the multiplicities of the polar divisor are not
    divisible by residue characteristics of \(U\).
  \end{itemize}
  We may assume $U$ is affine, and by abusing notation still
  denote its associated ring by $R$.
  For each closed point \(\mathfrak{p}\) in \(U\),
  let \(p\) be the characteristic of the residue field \(\kappa(\mathfrak{p})\).
  Then there exists a \(p\)-adically complete discrete valuation ring
  \(R'\) containing \(\zeta_p\) and a
  surjective ring homomorphism \(R \to R'\),
  such that the closed point of \(\mathrm{Spec}(R')\) is mapped to \(\mathfrak{p}\).
  Let \(K'\) be the field of
  fractions of \(R'\)  and let \(k'\) be its residue field.
  Then we have
  \begin{equation*}
    \mathrm{H}^{\bullet}_{\mathrm{rig}}(X_{\kappa(\mathfrak{p})}/K_{\mathfrak{p}},\mathcal{L}_{\pi}) \otimes_{K_{\mathfrak{p}}} K'
    \cong \mathrm{H}^{\bullet}_{\mathrm{rig}}(X_{k'}/K',\mathcal{L}_{\pi}).
  \end{equation*}
  We can then apply Theorem~\ref{theorem:main} to \(X_{R'}\),
  obtaining an isomorphism between the rigid cohomology
  \(\mathrm{H}^{\bullet}_{\mathrm{rig}}(X_{k'}/K',\mathcal{L}_{\pi})\) and
  a twisted de~Rham cohomology of \(X_{K'}\).

  To conclude, we use the following two facts:
  (i) \(\nabla_{c f}\) and \(\nabla_{f}\) have isomorphic de~Rham cohomology
  group over \(K'\), for any \(c \in K'^{\times}\); and
  (ii) the formation of twisted algebraic de~Rham cohomology is compatible with
  extension of scalars. The fact (i) is proved below as a lemma.
\end{proof}

\begin{lemma}%
  \label{lemma:constant-does-not-change-twisted-de-rham}
  Let \(K\) be a field of characteristic \(0\).
  Let \(f\colon X \to \mathbb{A}^{1}\) be a morphism of smooth \(K\)-schemes.
  Then for all \(c \in K^{\times}\), the dimensions of the \(K\)-vector spaces
  \(\mathrm{H}^{i}(X,\nabla_{cf})\)
  are the same.
\end{lemma}

\begin{proof}
  It suffices to prove \(\nabla_{cf}\) and \(\nabla_{f}\) have isomorphic
  twisted de~Rham cohomology groups. The standard argument (extracting
  coefficients defining \(X\) and \(f\), choosing a finitely generated subfield,
  embedding the field into \(\mathbb{C}\)) allows us to assume
  \(K = \mathbb{C}\).
  Then we can use the isomorphism provided by
  (\ref{sit:intro}/\ref{item:betti}) and (\ref{sit:intro}/\ref{item:alg-dr})
  to conclude that the two twisted de~Rham cohomology groups are isomorphic to the
  relative cohomology groups
  \begin{equation*}
    \mathrm{H}^{i}(X^{\mathrm{an}},f^{-1}(t)^{\mathrm{an}}), \quad \text{and} \quad
    \mathrm{H}^{i}(X^{\mathrm{an}},(cf)^{-1}(t)^{\mathrm{an}})
    \quad (|t| >\!\!> 0)
  \end{equation*}
  respectively. When \(|t|\) is large, these two groups are isomorphic, since
  \(f\) is topologically a fibration away from finitely many points.
\end{proof}

\section{Examples}
\label{sec:examples}

\begin{example}%
  \label{example:illustration-newton-degenerate}
  As a sanity check, here is a trivial example calculating the degree of
  the L-series of a Newton degenerate polynomial.

  Let \(k\) be a finite field of characteristic \(> 2\).
  Let \(f\colon \mathbb{A}^{2}_{k} \to \mathbb{A}^{1}_{k}\) be the regular
  function defined by the polynomial \(f(x,y) = x^{2}y-x\).
  (The projective completion of the curve \(f(x,y) = t\), \(t\neq 0\), has a
  cuspidal singularity at \([0,1,0]\), which can be simultaneously resolved by
  blowing this point up.)

  The morphism is smooth, but the polynomial \(f\) is neither Newton
  non-degenerate nor convenient, thus the theorems of
  Adolphson--Sperber~\cite{adolphson-sperber:newton-polyhedra-degree-l-function}
  and Denef--Loeser~\cite{denef-loeser:weights-exponential-sums-newton-polyhedra}
  do not give the degree of the L-series.

  A direct computation is easy:
  \begin{align*}
    S_{m}(f) &= \sum_{x,y \in k_m} \psi_m(x^2y-x)) \\
             &= \sum_{x,y\in k_m} \psi_m(x^{2}y)\psi_{m}(x)^{-1} \\
             &= q^{m} + \sum_{x\neq 0}\psi(x)^{-1}\sum_{y\in k_m} \psi_{m}(x^{2}y)   \\
             &= q^{m}.
  \end{align*}
  It follows that the L-series is \((1-qt)^{-1}\).

  Here is a topological calculation.
  It is easy to see \(f(x,y) = t\) for \(t \neq 0\) is
  isomorphic to \(\mathbb{G}_{\mathrm{m}}\) as schemes over \(k\). Thus one can
  still use the Teichm\"uller lift of \(f\) to conclude that the rigid
  cohomology \(\mathrm{H}^{\ast}(\mathbb{A}^{2}_{k},f^{\ast}\mathcal{L}_{\pi})\)
  is \(1\)-dimensional in degree \(2\), and zero otherwise. In particular,
  \(L_{f}(t)^{-1}\) is a linear polynomial.
  % We can also view \(f\) as a function
  % on \(\mathbb{G}_{\mathrm{m}}^{2}\).
\end{example}

Next, we recall the degree and total degree of the L-series of
exponential sums.
Let notation be as in Paragraph~(\ref{sit:intro}/\ref{item:et}).
In view of Grothendieck's theorem, we can write
\(L_{f}(t) = P(t)/Q(t)\), where \(P,Q \in \overline{\mathbb{Q}}[t]\) are
coprime. Then the \emph{degree}  of \(L_{f}\) is \(\deg Q - \deg P\);
the \emph{total degree} of \(L_{f}\) is \(\deg P + \deg Q\).

By Berthelot's trace formula for rigid cohomology, we know that the degree
of \(L_f\) is equal to the Euler characteristic
\(\sum (-1)^{i}\dim \mathrm{H}_{\mathrm{rig,c}}^{i}(X,f^{\ast}\mathcal{L}_{\pi})\),
and the total degree of \(L_{f}\) is no bigger than
\(\sum \dim \mathrm{H}_{\mathrm{rig,c}}^{i}(X,f^{\ast}\mathcal{L}_{\pi})\).
(Since \(X\) is smooth, Poincar\'e duality identifies the dimension of the
compactly supported cohomology with the dimension of the rigid cohomology up to
a dimension shift.)

\begin{example}%
  \label{example:b3-arrangement}
  Hyperplane arrangements give rise to many interesting exponential sums whose
  L-series cannot be computed using traditional methods. But topological and
  combinatorial methods sometimes can deduce useful information about the Milnor
  fiber of the arrangement, which, through Theorem~\ref{theorem:main}, can
  determine the dimension of the rigid cohomology for large primes.

  Here are two concrete examples.
  \begin{enumerate}[wide]
  \item Let \(f(x,y,z)=xyz(x-y)(y-z)(z-x)\) be the polynomial
    defining the so-called \(A_3\) plane arrangement.
    Then \(f\) is not convenient,
    but Cohen and Suciu~\cite[Example~5.1]{cohen-suciu:milnor-fibrations-of-arrangements}
    have shown that the Betti numbers of the ``Milnor fiber'' \(f=1\)
    are: \(b_0=1\), \(b_1=7\), \(b_2=18\).
    Since \(f\) is homogeneous, all the fibers \(f^{-1}(t)\), \(t\neq 0\), are homeomorphic.
    It follows that the relative cohomology
    \(\mathrm{H}^{\bullet}(\mathbb{C}^{3},f^{-1}(t))\)
    equals the \emph{reduced} singular cohomology of \(f^{-1}(1)\), up to a
    shift of cohomology degree.
    Using the isomorphism between relative cohomology
    (\ref{sit:intro}/\ref{item:betti})
    and twisted de~Rham cohomology~(\ref{sit:intro}/\ref{item:alg-dr}),
    we deduce that the nonzero Betti numbers of the twisted de~Rham cohomology are
    \(b_2=7, b_3=18\).
    For \(p\) large, Theorem~\ref{theorem:main-weak} implies that the rigid cohomology
    groups of \(f^{\ast}\mathcal{L}_{\pi}\) with compact support are of dimensions
    \(0, 7, 18\) respectively.  In particular, the degree of the L-series of the
    exponential sum associated with \(f\) has degree \(11\); and total degree
    \(\leq 25\).
    % Bombieri's bound of tot. deg. of a polyn in N var is (4 * deg f + 5)^N
  \item Let \(f(x,y,z) = xyz(x+y)(x-y)(x+z)(x-z)(y+z)(y-z)\)
    be the polynomial defining the so-called \(B_3\) plane arrangement.
    Cohen and Suciu~\cite[Example~5.2]{cohen-suciu:milnor-fibrations-of-arrangements}
    has shown that the Betti numbers of the Milnor fiber \(f=1\)
    are: \(b_0=1\), \(b_1=8\), \(b_2=79\). As in (1), when \(p\) is large,
    we get the dimension of the rigid cohomology of \(f^{\ast}\mathcal{L}_{\pi}\), and
    we conclude that the degree of the L-series of \(f\) is \(71\), and the
    total degree is \(\leq 87\).
  \end{enumerate}
\end{example}

\begin{example}[Exponential sum on a curve]%
  \label{example:illustration-curve}
  Let \(X\) be a smooth irreducible projective curve of genus \(g\)
  over a finite field \(k\).
  Let \(f\colon X \to \mathbb{P}^{1}_{k}\) be a generically étale, degree \(d\) morphism with
  ramification indices not divisible by \(p\). We assume there is a good lift of
  \(f\) (i.e., satisfies Hypothesis~\ref{situation:general-case}).
  Let \(Z = f^{-1}(\infty) = \{\tau_1,\ldots,\tau_m\}\).
  Let \(V\) be a nonempty Zariski open subset of \(X\) such that
  \(f(V) \subset \mathbb{A}^{1}_{k}\). We assume that the ramification points of
  \(f\) as well as the points \(X\setminus V\) are rational over \(k\) (which is
  harmless in considering the degree of L-series). Let
  \(c = \operatorname{Card}(Z\setminus V)\).
  % The ramification index of \(f\) at \(\tau_i\) is denoted by \(d_i\).

  Then a topological calculation implies that the twisted algebraic de~Rham
  cohomology of a lift of \(f\) has cohomology in degree \(1\) only, and the
  dimension equals \(2g+c+m+d-2\). In this case Theorem~\ref{theorem:main-weak}
  applies and we conclude that \(L_{f|_{C}}^{-1}(t)\) is a polynomial of degree
  \(2g+c+m+d-2\). This matches the length of the ``Hodge polygon'' considered by
  Kramer-Miller~\cite[Theorem~1.1]{kramer-miller:newton-above-hodge-for-curves}.
\end{example}

\begin{example}[Exponential sum on \(\mathrm{SL}_{2}\)]%
  \label{example:illustration-sl2}
  Let \(k\) be a finite field of characteristic \(p>0\).
  Let \(V = k^2\) be the standard representation of \(\mathrm{SL}_{2}\).
  For \(A \in \mathrm{SL}_{2}(k)\), let \(A^{(n)}=\mathrm{Tr}(\mathrm{Sym}^{n}A)\).
  Then for \(a_1,\ldots,a_{N} \in k\),
  \begin{equation*}
    f(A) = \sum_{n=1}^{N} a_{n}A^{(n)}
  \end{equation*}
  is a regular function on \(\mathrm{SL}_{2}\).
  Then the rigid cohomology
  \(\mathrm{H}^{\ast}_{\mathrm{rig}}(\mathrm{SL}_{2},f^{\ast}\mathcal{L}_{\pi})\)
  is related to the exponential sum
  \begin{equation}
    \label{eq:exp-sum-on-sl2}
    \sum_{A \in \mathrm{SL}_2(k_m)} \psi_m(f(A)),
  \end{equation}
  where \(k_m\) is a fixed degree \(m\) extension of \(k\),
  \(\psi_1\) is a nontrivial additive character on \(k\),
  \(\psi_m = \psi_1\circ \mathrm{Tr}_{k_m/k}\).

  If \(p\) is sufficiently large, and if \((a_1,\ldots,a_N)\) is sufficiently general, we
  can calculate the dimension of the rigid cohomology using topology.
  It is not difficult to see \(f^{-1}(t)\) is \(N\) disjoint union of
  \(\mathbb{P}^{1}\times \mathbb{P}^{1}\setminus \Delta\).
  Using the long exact sequence for relative cohomology one sees that
  \(\mathrm{H}^{\ast}(\mathrm{SL}_2(\mathbb{C}),f^{-1}(t))\) is nonzero only in
  degree \(1\), and \(3\) and of dimension \(N-1\), \(N+1\) respectively.
  Thus the L-series \(L_f(t)\) of the exponential sums~\eqref{eq:exp-sum-on-sl2}
  is the reciprocal of a degree \(2N\) polynomial.
\end{example}

\begin{example}
  \label{example:last}
  Below we shall prove
  Corollary~\ref{corollary:standard-situation}, confirming the assertions
  made in Example~\ref{example:how-to-produce-such-a-function}.

  We begin with some rather trivial topological discussions.
  In (\ref{lemma:relative-cohomology-with-fiber}) -- (\ref{lemma:lefschetz}),
  complex algebraic varieties are equipped with analytic topology,
  and singular cohomology groups are assumed to have coefficients in
  \(\mathbb{Q}\).

  Let \(f\colon \overline{X} \to \mathbb{A}^{1}\)
  be a proper, generically smooth morphism of algebraic varieties over \(\mathbb{C}\).
  Let \(X\) be an open subvariety of \(\overline{X}\), and \(E = \overline{X}\setminus X\).
  Suppose that there is a neighborhood \(T\) of \(E\) such that \(f|_T\) and
  \(f|_{T\setminus E}\) are  locally topologically trivial fibrations
  (hence are trivial as \(\mathbb{A}^{1}\) is contractible).
  Let \(\overline{F}\) be a generic fiber of \(f\)
  and \(F = \overline{F} \cap X\).

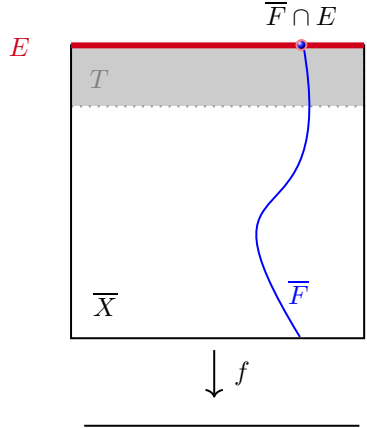
      \begin{figure*}[ht]
      \centering
      \tikzset{every picture/.style={line width=0.75pt}} %set default line width to 0.75pt

      \begin{tikzpicture}[x=0.75pt,y=0.75pt,yscale=-1,xscale=1]
        % uncomment if require: \path (0,300); %set diagram left start at 0, and has height of 300

        % Shape: Rectangle [id:dp8091904029643517]
        \draw  [color={rgb, 255:red, 155; green, 155; blue, 155 }  ,draw opacity=1 ][fill={rgb, 255:red, 155; green, 155; blue, 155 }  ,fill opacity=0.51 ][dash pattern={on 0.84pt off 2.51pt}] (62.8,16.2) -- (210.62,16.2) -- (210.62,46.98) -- (62.8,46.98) -- cycle ;
        % Shape: Square [id:dp5101375395281176]
        \draw   (62.8,16.2) -- (210.62,16.2) -- (210.62,164.02) -- (62.8,164.02) -- cycle ;
        % Straight Lines [id:da9422319890341557]
        \draw    (69.42,208.2) -- (208.22,208.2) ;
        % Straight Lines [id:da3042679873381622]
        \draw [color={rgb, 255:red, 208; green, 2; blue, 27 }  ,draw opacity=1 ][line width=2.25]    (62.8,16.2) -- (210.62,16.2) ;
        % Straight Lines [id:da03592418759874083]
        \draw[->]    (135,170) -- (135,194) ;
        \draw    (149,182) node{\(f\)} ;
        % Curve Lines [id:da37681124618411954]
        \draw [color=blue  ,draw opacity=1 ]   (179.87,15.8) .. controls (199.29,124.58) and (119.42,66.98) .. (178.22,163.38) ;
        % Shape: Circle [id:dp6860939275689986]
        \shadedraw[ball color=blue, draw=red!50] (179,16) circle (2.5);
        \draw (179,0) node {\(\overline{F} \cap E\)};
        % Text Node
        \draw (30,11.2) node [anchor=north west][inner sep=0.75pt]  [color={rgb, 255:red, 208; green, 2; blue, 27 }  ,opacity=1 ]  {$E$};
        % Text Node
        \draw (70.8,27) node [anchor=north west][inner sep=0.75pt][color=gray]    {$T$};
        % Text Node
        \draw (170.8,133.6) node [anchor=north west][inner sep=0.75pt]  [color=blue  ,opacity=1 ]  {$\overline{F}$};
        % Text Node
        \draw (72.4,139.6) node [anchor=north west][inner sep=0.75pt]    {$\overline{X}$};
      \end{tikzpicture}
    \end{figure*}

  \begin{lemma}%
    \label{lemma:relative-cohomology-with-fiber}
    Notation as above, the natural map
    \(\mathrm{H}^{i}(\overline{X},\overline{F}) \to \mathrm{H}^{i}(X,F)\)
    is an isomorphism for all \(i\).
  \end{lemma}

  \begin{proof}
    This is a simple application of excision.
    To begin with, since \(T \to \mathbb{A}^{1}\) is a trivial fibration,
    there is a deformation retract from \(T\) onto \(T \cap \overline{F}\).
    This induces a deformation retract from \(T \cup \overline{F}\) onto
    \(\overline{F}\).
    Thus the pair \((\overline{X}, T \cup \overline{F})\) and the pair
    \((\overline{X},\overline{F})\) have the same cohomology.

    Since \(E\) is contained in the interior of \(T \cup \overline{F}\),
    excision implies that the pair \((\overline{X},T \cup \overline{F})\) and
    the pair \((X, (T\setminus E) \cup F)\) have the same cohomology.
    Using the fact that \(T \setminus E \to \mathbb{A}^{1}\)  is a trivial
    fibration, we find a deformation retract from \((T\setminus E) \cup F\)
    onto \(F\). Thus the pair \((X,(T\setminus E)\cup F)\) and the pair
    \((X,F)\) have the same cohomology. This completes the proof.
  \end{proof}

  Let \(P\) be a smooth projective variety.
  Let \(\mathcal{L}_{1}, \ldots, \mathcal{L}_{r}\)
  be invertible sheaves on \(P\).
  Fix sections \(s_i \in \mathrm{H}^{0}(P,\mathcal{L}_{i})\),
  whose associated Cartier divisor is \(D_{i}\).
  Assume that \(D_i\) is smooth and irreducible.
  Let \(e_1,\ldots,e_r\) be non-negative integers.
  Let \(s \in \mathrm{H}^{0}\left(P,\bigotimes_{i=1}^{r}\mathcal{L}_{i}^{\otimes e_i}\right)\).
  Denote by \(X_0\) the Cartier divisor defined by \(s\), and assume that \(X_0\)
  is smooth and irreducible.
  Let \(X_{\infty} = \sum e_{i}D_{i}\).
  Assume that \(X_0 + X_{\infty}\) is a divisor with strict normal crossings.
  Let \(\overline{X} = \mathrm{Bl}_{Z}P \setminus X_{\infty}\), where \(Z\) is the
  base locus of the pencil generated by \(X_0\) and \(X_{\infty}\). Then there is
  a generically smooth morphism \(f\colon \overline{X} \to \mathbb{A}^{1}\)
  and a projection \(\pi \colon \overline{X} \to P\).

  Now we are in the situation considered in
  Lemma~\ref{lemma:relative-cohomology-with-fiber}. Here \(E\) is the
  intersection of the exceptional divisor in \(\mathrm{Bl}_{Z}P\) with
  \(\overline{X}\). By construction, it has a ``tubular neighborhood'' \(T\)
  (e.g., the preimage of a tubular neighborhood of \(Z\) under the blowing up)
  within \(\overline{X}\) such that the restriction of \(f\) to \(T\) is a
  topologically trivial fibration.

  Retain the above notation. Let \(s_{\infty} = \prod s_{i}^{e_i}\).
  Let \(X = P \setminus X_{\infty}\).
  Then via the morphism \(\pi\), we can regard \(X\) as an open subscheme of
  \(\overline{X}\). The restriction of \(f\) to \(X\) is given by the ratio
  \(g = s_0/s_{\infty}\). Then by
  Lemma~\ref{lemma:relative-cohomology-with-fiber}, for \(t\) generic
  we have
  \begin{equation*}
    \mathrm{H}^{i}(\overline{X},X_t) \cong \mathrm{H}^{i}(X,g^{-1}(t)) \cong
    \mathrm{H}^{i}(\overline{X},g^{-1}(0)).
  \end{equation*}
  In practice, we are more interested in considering the function \(g\) on
  \(X\); and the construction above allows us to construct a proper function
  which is ready for taking reduction modulo \(p\).

  The next lemma tells us that in a certain preferable situation,
  the calculation of cohomology groups
  reduces to the calculation of the Euler characteristics.

  \begin{lemma}%
    \label{lemma:lefschetz}
    Notation as above, assume in addition that \(X\) and \(P \setminus X_{0}\) are affine
    (e.g., when the invertible sheaf \(\bigotimes_{i=1}^{r}\mathcal{L}_{i}^{\otimes e_i}\)
    is ample).
    Then \(\mathrm{H}^{i}(X,g^{-1}(0))\) is nonzero only if \(i = \dim X\).
  \end{lemma}

  The Euler characteristic of
  \(\mathrm{H}^{\bullet}(X,g^{-1}(0))\)
  is
  \begin{equation*}
    (-1)^{\dim P}\int_P \frac{c((\Omega^{1}_P)^{\vee})}{(1 + \sum_{i=1}^{r} e_{i}c_{1}(\mathcal{L}_{i}))\prod_{i=1}^{r}(1+c_1(\mathcal{L}_{i}))}.
  \end{equation*}
  See for example~\cite[Theorem~5.4.1]{katz:sommes-exponentielles}.
  If the hypothesis of Lemma~\ref{lemma:lefschetz} is fulfilled,
  then the absolute value of
  this number is also the dimension of \(\mathrm{H}^{\dim X}(X,g^{-1}(0))\).

  \begin{proof}[Proof of Lemma~\ref{lemma:lefschetz}]
    Since \(X\) and \(g^{-1}(0)\) are smooth affine varieties,
    the relative cohomology \(\mathrm{H}^{i}(X,g^{-1}(0))\) vanishes
    if \(i> \dim X\).
    It suffices to prove the relative cohomology also vanishes when
    \(i < \dim X\).

    For any subset \(J\) of \(\{1,2,\ldots,r\}\),
    let \(D_{J} = \bigcap_{j\in J}D_{j}\).
    Let \(D^{(p)} = \coprod_{\operatorname{Card}J=p}D_{J}\), and write \(D^{(0)}=P\).
    The scheme \(D^{(p)}\) is smooth proper of dimension \(n-p\), and
    the natural morphism \(D^{(p)} \to P\) is affine.

    There exists a spectral sequence
    \begin{equation}\label{eq:spectral-sequence}
      E_{1}^{-p,q} = \mathrm{H}^{q-2p}(D^{(p)},D^{(p)}\times_{P} X_0)
      \Rightarrow \mathrm{H}^{q-p}(X,g^{-1}(0)).
    \end{equation}
    Granting the existence of this spectral sequence, let us finish the proof.
    Since \(P\setminus X_0\) is affine,
    \(D^{(p)} \setminus D^{(p)} \times_P X_0\) are affine for all \(p\).
    If \(i < \dim X\), then \(i - p < \dim D^{(p)}\).
    By Artin's vanishing theorem, we have
    \[
      \mathrm{H}^{i-p}(D^{(p)},D^{(p)}\times_P X_0) =
      \mathrm{H}^{i-p}_{c}(D^{(p)}\setminus D^{(p)}\times_P X_0) = \{0\}.
    \]
    It follows that \(E_{1}^{-p,q}=0\) if \(q-p< \dim X\).
    This implies that \(\mathrm{H}^{i}(X,g^{-1}(0))\) vanishes when
    \(i<\dim X\), as desired. The construction of the spectral
    sequence~\eqref{eq:spectral-sequence} will be recalled in
    Paragraph~\ref{situation:spectral} at the end of this section.
  \end{proof}

  Having discussed some topology, we now return to
  Example~\ref{example:how-to-produce-such-a-function}. We henceforth enforce
  the notation set up there.

  \begin{corollary}%
    \label{corollary:standard-situation}
    Assume \(\mathfrak{p}\) is a prime of
    \(\mathbf{K}\) satisfying the two conditions
    (\ref{example:how-to-produce-such-a-function}/1) and
    (\ref{example:how-to-produce-such-a-function}/2).
    Let \(p\) be the residue characteristic of \(\mathfrak{p}\),
    and \(\mathbf{K}_{\mathfrak{p}}\) be the completion of \(\mathbf{K}\) at
    \(\mathfrak{p}\).
    \begin{enumerate}[wide]
    \item
      The dimension of the rigid cohomology
      \(\mathrm{H}^{i}(X\otimes\mathbf{K}_{\mathfrak{p}}(\zeta_p),\mathcal{L}_{\pi})\)
      equals the dimension of
      \(\mathrm{H}^{i}(X^{\mathrm{an}}_{\mathbb{C}},g^{-1}(0)_{\mathbb{C}}^{\mathrm{an}})\).
    \item If both \(X\) and \(P \setminus X_0\) are affine,
      the rigid cohomology is nonzero only in cohomology degree \(\dim X\),
      and its dimension over \(\mathbf{K}_{\mathfrak{p}}\) is
      \begin{equation*}
        \int_P \frac{c((\Omega^{1}_P)^{\vee})}{(1 + \sum_{i=1}^{r} e_{i}c_{1}(\mathcal{L}_{i}))\prod_{i=1}^{r}(1+c_1(\mathcal{L}_{i}))}.
      \end{equation*}
      Thus, the L-series associated with the function \(g\) modulo
      \(\mathfrak{p}\) is a polynomial or a reciprocal of a polynomial,
      whose degree equals
      \(
      \int_P \frac{c((\Omega^{1}_P)^{\vee})}{(1 + \sum_{i=1}^{r} e_{i}c_{1}(\mathcal{L}_{i}))\prod_{i=1}^{r}(1+c_1(\mathcal{L}_{i}))}
      \).
    \end{enumerate}
  \end{corollary}

  If \(\mathcal{L}_i\) are ample, the assertion (2) is due to
  N.~Katz~\cite{katz:sommes-exponentielles}.

  \begin{proof}[Proof of Corollary~\ref{corollary:standard-situation}]
    We could assume that \(\mathbf{K}\) contains a \(p\)\textsuperscript{th}
    root of unity. The morphism
    \[
      f\colon \mathrm{Bl}_{Z}P \otimes \mathcal{O}_{\mathbf{K}_{\mathfrak{p}}}
      \to \mathbb{P}^{1}_{\mathcal{O}_{\mathbf{K}_{\mathfrak{p}}}}
    \]
    then satisfies the hypothesis of Theorem~\ref{theorem:main}.
    Thus the theorem implies the rigid cohomology associated with the reduction
    of \(g\) has the same dimension as the twisted algebraic de~Rham cohomology
    over \(\mathbf{K}_{\mathfrak{p}}\).

    By performing a base extension to \(\mathbb{C}\), using the isomorphism provided
    by~(\ref{sit:intro}/\ref{item:betti}) and (\ref{sit:intro}/\ref{item:alg-dr}),
    we know that the dimension of the twisted algebraic de~Rham cohomology
    defined by \(g\) is equal to the dimension of the relative cohomology
    \(\mathrm{H}^{i}(X^{\mathrm{an}}_{\mathbb{C}},g^{-1}(0)_{\mathbb{C}}^{\mathrm{an}})\).
    Here using \(t=0\) instead of a generic \(t\) is legal,
    because by construction \(0\) is a typical value of \(g\); see the footnote on
    page~\pageref{item:betti}.

    The second assertion follows from Lemma~\ref{lemma:lefschetz} and the
    Frobenius trace formula.
  \end{proof}

  A particular instance of the standard situation is as follows.
  Let \(P=\mathbb{P}^{n}\), \(X_{\infty}\) be the union of coordinate lines defined by
  \(z_0 z_1 \cdots z_n = 0\), and let \(X_0\) be the Fermat hypersurface of degree
  \(n+1\). Then the function
  \[
    z \mapsto \frac{z_0^{n+1} + z_1^{n+1} + \cdots + z_n^{n+1}}{z_0\cdots z_n}
    \colon\mathbb{G}_{\mathrm{m}}^{n} \to \mathbb{A}^{1}
  \]
  fits the standard situation. Assuming the residue characteristic of
  \(\mathfrak{p}\) is not a factor of \(n+1\),  then we can apply Theorem~\ref{theorem:main}.
  Since this is the ``ample case'',  Lemma~\ref{lemma:lefschetz} implies that the
  rigid cohomology is trivial except in degree \(n\), and its dimension equals
  \((-1)^{n}\) times the Euler characteristics, which is \(n^{n}(n+1)\).
\end{example}

% \begin{example}%
%   \label{example:illustration-method}
%   Let \(k\) be a finite field. Let \(u \colon S \to \mathbb{P}^{1}_{k}\) be a
%   Lefschetz pencil of cubic curves intersecting at \(9\) points rational over
%   \(k\). Assume that the member over \(\infty \in \mathbb{P}^{1}\) is
%   nonsingular, and let \(f\colon S'=S \setminus u^{-1}(\infty) \to
%   \mathbb{A}^{1}_k\) be the restriction of \(u\).

%   By lifting this pencil to characteristic \(0\), Theorem~\ref{theorem:main}
%   implies that the rigid cohomology and the algebraic de~Rham cohomology
%   agree. In this case, the algebraic de~Rham cohomology can be calculated
%   topologically: it is the sum of the vanishing cycles in the finite vicinity.
%   As each of the node will contribute \(1\) to the total number of vanishing
%   cycles, we conclude that
%   \(\mathrm{H}^{\bullet}_{\mathrm{rig}}(S'/K,\mathcal{L}_{\pi})\) is nonzero
%   only in degree \(2\), and the degree \(2\) piece is a 12 dimensional vector
%   space over \(K\). In particular, the L-series of \(f\) is a polynomial of
%   degree \(12\).
% \end{example}

\begin{situation}\label{situation:spectral}%
    We recall a construction of the spectral
    sequence~\eqref{eq:spectral-sequence}
    for convenience of the reader. Notation and conventions will be as in
    Lemma~\ref{lemma:lefschetz}.

    Consider the natural simplicial morphism \(a\colon D^{(\bullet+1)} \to P\),
    viewing \(P\) as a constant simplicial scheme. Then the relative cohomology
    of the pair \((X,g^{-1}(0))\) is computed by the homotopy cofiber of
    the morphism
    \begin{equation*}
      a_{\ast}a^{!} v_{!}\mathbb{Q} \to v_{!}\mathbb{Q},
    \end{equation*}
    where \(v\colon P\setminus X_0 \to P\) is the open immersion.
    
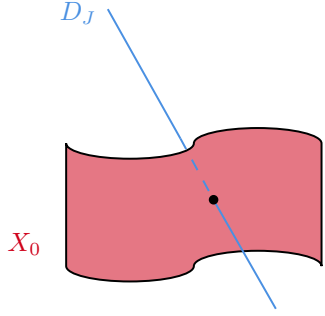
\begin{figure}[ht]
      \centering
      \tikzset{every picture/.style={line width=0.75pt}} %set default line width to 0.75pt
      \begin{tikzpicture}[x=0.75pt,y=0.75pt,yscale=-1,xscale=1]
        % uncomment if require: \path (0,300); %set diagram left start at 0, and has height of 300
        % Flowchart: Punched Tape [id:dp7437258078956803]
        \draw  [fill={rgb, 255:red, 208; green, 2; blue, 27 }  ,fill opacity=0.54 ] (70,141.74) .. controls (70,146.02) and (84.42,149.49) .. (102.2,149.49) .. controls (119.98,149.49) and (134.4,146.02) .. (134.4,141.74) .. controls (134.4,137.47) and (148.82,134) .. (166.6,134) .. controls (184.38,134) and (198.8,137.47) .. (198.8,141.74) -- (198.8,203.7) .. controls (198.8,199.42) and (184.38,195.96) .. (166.6,195.96) .. controls (148.82,195.96) and (134.4,199.42) .. (134.4,203.7) .. controls (134.4,207.98) and (119.98,211.45) .. (102.2,211.45) .. controls (84.42,211.45) and (70,207.98) .. (70,203.7) -- cycle ;
        % Straight Lines [id:da2061839561003893]
        \draw [color={rgb, 255:red, 74; green, 144; blue, 226 }  ,draw opacity=1 ]   (91,74) -- (127.69,139.41) ;
        % Straight Lines [id:da17256905062915995]
        \draw [color={rgb, 255:red, 74; green, 144; blue, 226 }  ,draw opacity=1 ] [dash pattern={on 4.5pt off 4.5pt}]  (127.69,139.41) -- (145.38,172.12) ;
        % Straight Lines [id:da3280163462599892]
        \draw [color={rgb, 255:red, 74; green, 144; blue, 226 }  ,draw opacity=1 ]   (145.38,172.12) -- (175.76,225.24) ;
        % Shape: Circle [id:dp6558663758131058]
        \draw  [fill={rgb, 255:red, 0; green, 0; blue, 0 }  ,fill opacity=1 ] (142.44,170.24) .. controls (142.43,169.18) and (143.27,168.31) .. (144.32,168.29) .. controls (145.38,168.28) and (146.25,169.12) .. (146.26,170.18) .. controls (146.28,171.23) and (145.44,172.1) .. (144.38,172.12) .. controls (143.33,172.13) and (142.46,171.29) .. (142.44,170.24) -- cycle ;

        % Text Node
        \draw (38.86,185.4) node [anchor=north west][inner sep=0.75pt]  [color={rgb, 255:red, 208; green, 2; blue, 27 }  ,opacity=1 ]  {$X_{0}$};
        % Text Node
        \draw (65.24,68.8) node [anchor=north west][inner sep=0.75pt]  [color={rgb, 255:red, 74; green, 144; blue, 226 }  ,opacity=1 ]  {$D_{J}$};
      \end{tikzpicture}
    \end{figure}

    The complex
    \[
      a_{\ast}a^{!}v_{!}\mathbb{Q} = [\cdots \to a_{2\ast}a_{2}^{!}v_{!}\mathbb{Q} \to a_{1\ast}a_{1}^{!}v_{1}\mathbb{Q}]
    \]
    can be understood explicitly using the fact that
    \(D_{J}\) intersects \(X_0\) transversely. Indeed, if \(a_{J}\) is the
    inclusion morphism from \(D_{J}\) into \(P\), then \(a_{i} = \coprod_{\operatorname{Card}J=i}a_{J}\),
    and (we use \(\vee\) for Verdier dual)
    \begin{equation*}
      a_{J\ast}a_{J}^{!} v_{!}\mathbb{Q}[\dim X] = a_{J\ast}(a_{J}^{\ast}v_{\ast}\mathbb{Q}[\dim X])^{\vee}.
    \end{equation*}
    Let \(v_{J}\colon D_J \setminus D_{J} \cap X_0 \to D_{J}\) be the inclusion morphism.

    We claim that \(a_{J}^{\ast}Rv_{\ast}\mathbb{Q} \cong Rv_{J\ast}\mathbb{Q}\).
    Indeed, consider
    the following fiber diagram,
    \begin{equation*}
      \begin{tikzcd}
        X_0 \cap D_J \ar{r} \ar{d}{\iota_J} & X_0 \ar{d}{\iota} \\
        D_{J} \ar{r}{a_J} & P
      \end{tikzcd},
    \end{equation*}
    and the distinguished triangle
    \begin{equation*}
      \iota_{\ast}\iota^{!}\mathbb{Q} \to \mathbb{Q} \to Rv_{\ast} \mathbb{Q} \to.
    \end{equation*}
    Since \(X_0\) is a smooth divisor in \(P\), the Thom isomorphism theorem
    implies that \(\iota^{!}\mathbb{Q}_{P} = \mathbb{Q}_{X_0}[-2]\). Pulling back
    the above distinguished triangle by \(a_{J}\) yields a distinguished
    triangle
    \begin{equation*}
      a_{J}^{\ast}\iota_{\ast} \mathbb{Q}_{X_0}[-2] \to \mathbb{Q}_{D_J} \to a_{J}^{\ast} Rv_{\ast}\mathbb{Q} \to.
    \end{equation*}
    Since \(\iota\) is proper, and since \(X_0 \cap D_J\) is smooth of
    codimension \(1\), we have
    \begin{align*}
      a_{J}^{\ast} \iota_{\ast} \mathbb{Q}_{X_0}[-2]
      &=\iota_{J\ast} \mathbb{Q}_{X_0\cap D_J}[-2] \\
      &=\iota_{J\ast}\iota_{J}^{!}\mathbb{Q}_{D_J}.
    \end{align*}
    By adjunction, (Hom being taken in the derived category)
    \(\mathrm{Hom}(\iota_{J\ast}\iota_{J}^{!}\mathbb{Q}_{X_0},\mathbb{Q}_{X_0})\)
    equals
    \(\mathrm{Hom}(\iota_{J}^{!}\mathbb{Q},\iota_{J}^{!}\mathbb{Q})=\mathrm{H}^{0}(X_0\cap D_{J},\mathbb{Q})\).
    Thus up to a nonzero constant multiple on each connected component of \(X_0 \cap D_J\),
    there exists only one nonzero morphism from \(\iota_{J\ast}\iota_{J}^{!}\mathbb{Q}\)
    to \(\mathbb{Q}\) in the derived category. (In fact, using the existence of
    integral structure, the term ``up to a nonzero constant'' could be replaced
    by ``up to sign''.)
    Therefore the homotopy cofiber of the morphism
    \begin{equation*}
      \begin{tikzcd}[row sep=small]
        a_{J}^{\ast}\iota_{\ast} \mathbb{Q}_{X_0}[-2] \ar[equal]{d} \ar{r} & \mathbb{Q}_{D_J}  \\
        \iota_{J\ast}\iota_{J}^{!}\mathbb{Q}_{D_J}
      \end{tikzcd}
    \end{equation*}
    is necessarily isomorphic to the homotopy cofiber of the canonical morphism
    \[\iota_{J\ast}\iota_{J}^{!}\mathbb{Q}_{D_J} \to \mathbb{Q}_{D_J}\]
    which is \(Rv_{J\ast}\mathbb{Q}\). This justifies the claim.

    Thus, letting
    \(p=\operatorname{Card}J\) be the codimension of \(D_{J}\), we have
    \begin{equation*}
      a_{J\ast}a_{J}^{!} v_{!}\mathbb{Q} =
      a_{J\ast}(v_{J!} \mathbb{Q})[-2p].
    \end{equation*}
    The complex \(a_{J\ast}(v_{J!} \mathbb{Q})\) is precisely the derived
    incarnation of the relative cohomology
    \(\mathrm{H}^{\bullet}(D_{J},D_{J}\cap X_0)\).
    The spectral sequence (associated with the ``filtration bête'') of the
    complex \(a_{\ast}a^{!}v_{!}\mathbb{Q}\to \mathbb{Q}\) gives the
    spectral sequence~\eqref{eq:spectral-sequence}.
\end{situation}

\section{A remark on Higgs cohomology}

In the theory of twisted algebraic de~Rham cohomology, one can associate an
irregular Higgs field to the connection \(\nabla_{f}\), i.e., the Higgs field
defined by \(\wedge \mathrm{d}f\). A celebrated theorem of
Barannikov--Kontsevich (the first published proof is due to
Sabbah~\cite{sabbah:twisted-de-rham}) asserts that the algebraic Higgs cohomology and the
twisted algebraic de~Rham cohomology have the same dimension.

In this section we explain how to transplant this to the rigid analytic world.
We shall prove a comparison between a ``dagger'' Higgs
cohomology and the algebraic Higgs cohomology. Then in the nice situation,
Theorem~\ref{theorem:main-weak} allows us to relate twisted rigid cohomology and
the dagger Higgs cohomology.

In the work of
Adolphson--Sperber~\cite{adolphson-sperber:newton-polyhedra-degree-l-function}
on exponential sums, one also finds the use of
Higgs cohomology. In fact their finiteness theorem is deduced from the
finiteness of the Higgs cohomology of the reduction. In contrast, the result of
this section happens completely on the generic fiber, thereby does not really
concern whether the pole divisor has good reduction or not.

\begin{situation}\label{situation:hypothesis-higgs}%
  \textbf{Notation.}
  \begin{itemize}[wide]
  \item
    Let \(K\) be a discrete valuation field of characteristic \(0\) whose
    residue characteristic is \(p\).
    Let \(\mathfrak{X}\) be a scheme smooth over a discrete valuation ring \(\mathcal{O}_{K}\).
  \item   Let \(f\colon \mathfrak{X} \to \mathbb{A}^{1}_{R}\) be a \emph{proper} function
    admitting a compactification
    \(\overline{f}\colon \overline{\mathfrak{X}} \to \mathbb{P}^{1}_{\mathcal{O}_K}\),
    in which \(\overline{\mathfrak{X}}\) is smooth over \(R\).
  \item Also denote by \(\overline{f}\colon \overline{X} \to \mathbb{P}^{1}_{K}\)
    and \(f\colon X \to \mathbb{P}^{1}_K\) be the restriction of \(f\) to the
    generic fibers \(\overline{X}\) and \(X\) of \(\overline{\mathfrak{X}}\) and
    \(\mathfrak{X}\) respectively.
  \item
    Let \(V_{r}\) be the inverse image \(f^{-1}(\mathbb{D}^{+}(0;r))\) of the
    rigid analytic disk under \(f\). Thus \(V_r\) is an rigid analytic subspace
    of \(\overline{X}\).
  \end{itemize}
\end{situation}

\begin{situation}
  \textbf{Hypothesis.}
  We assume that \(f\colon X \to \mathbb{A}^{1}_{K}\) has no critical values in
  \(\mathbb{D}^{-}(\infty;1)\).
\end{situation}

Note that we do not enforce the hypothesis that the components of the pole
divisor have good reduction anymore.

\begin{situation}
  We can consider three types of Higgs cohomology.
  \begin{enumerate}[wide]
  \item   The algebraic Higgs cohomology
    \begin{equation*}
      R\Gamma(X, (\Omega_{X}^{\bullet},\mathrm{d}f))
    \end{equation*}
    of \(X\).
  \item A dagger version of the Higgs
    cohomology
    \begin{equation*}
      R\Gamma(\overline{X}^{\mathrm{an}}, j_1^{\dagger}(\Omega_{\overline{X}^{\mathrm{an}}}^{\bullet},\mathrm{d}f)),
    \end{equation*}
    where for a real number \(r\),
    \(j_{r}\colon V_r \to \overline{X}\) denotes the open inclusion.
  \item The analytic Higgs cohomology
    \begin{equation*}
      R\Gamma(V_{r}, (\Omega_{V_{r}}^{\bullet},\mathrm{d}f)) \quad(r>1).
    \end{equation*}
  \end{enumerate}
\end{situation}

\begin{proposition}%
  \label{proposition:overconvergent-higgs-qis-algebraic-higgs}
  Notation as above, the natural morphisms
  \begin{equation*}
    R\Gamma(X, (\Omega_{X}^{\bullet},\mathrm{d}f)) \to
    R\Gamma(V_{r}, (\Omega_{V_{r}}^{\bullet},\mathrm{d}f))\to
    R\Gamma(\overline{X}^{\mathrm{an}}, j_{1}^{\dagger}(\Omega_{\overline{X}^{\mathrm{an}}}^{\bullet},\mathrm{d}f))
  \end{equation*}
  are quasi-isomorphisms.
\end{proposition}

\begin{proof}
  It suffices to prove the first arrow is a quasi-isomorphism,
  as the third item is obtained from the second by taking colimit with respect
  to \(r \to 1^{-}\).

  Let \(P\) be the pole divisor of
  \(f\colon \overline{X} \to \mathbb{P}^{1}_K\).
  For any connected subset \(I\) of \(\mathbb{R}_{\geq 0}\),
  let \(T_{I}\) be the inverse image of the rigid analytic annulus
  \(\Delta_{I}(\infty)\) centered at \(\infty \in \mathbb{P}^{1}\). Thus
  \begin{equation*}
    T_{I} = f^{-1}(\Delta_{I}(\infty)).
  \end{equation*}
  Let \(\Omega_{\overline{X}^{\mathrm{an}}}^{\bullet}(\ast P)\)
  be the subcomplex of \(j_{\ast}\Omega_{X^{\mathrm{an}}}^{\bullet}\) consisting of
  differential forms with at worst poles along \(P\).
  Then by rigid analytic GAGA, the natural morphism of complexes
  \begin{equation*}
    R\Gamma(X, (\Omega_{X},\mathrm{d}f)) \to
    \underbrace{\Gamma(\overline{X}^{\mathrm{an}},(\Omega_{\overline{X}^{\mathrm{an}}}^{\bullet}(\ast P),\mathrm{d}f))}_{\text{``moderate Higgs complex''}}
  \end{equation*}
  is a quasi-isomorphism.

  Choose a function
  \(r\mapsto \delta_r\colon \interval[open]{1}{\infty} \to \mathbb{R}\)
  such that \(1 > \delta_{r} > r^{-1}\).
  For each \(r > 1\), \(V_{r}\) and \(T_{[0,\delta_{r}]}\)
  form an admissible cover of \(\overline{X}\).
  By Mayer--Vietoris, the complex
  \(R\Gamma(\overline{X}^{\mathrm{an}},(\Omega_{\overline{X}^{\mathrm{an}}}^{\bullet}(\ast{P}),\mathrm{d}f))\)
  is the homotopy kernel of
  \begin{equation*}
    R\Gamma(V_{r},(\Omega_{V_r}^{\bullet},\mathrm{d}f)) \oplus
    R\Gamma(T_{[0,\delta_r]},(\Omega_{T_{[0,\delta_r]}}^{\bullet}(\ast{P}),\mathrm{d}f))
    \to
    R\Gamma(T_{[r^{-1},\delta_r]},(\Omega_{T_{[r^{-1},\delta_r]}}^{\bullet},\mathrm{d}f)))
  \end{equation*}
  We shall show that the natural morphism
  \begin{equation*}
    R\Gamma(T_{[0,\delta_r]},(\Omega_{T_{[0,\delta_r]}}^{\bullet}(\ast{P}),\mathrm{d}f))
    \to
    R\Gamma(T_{[r^{-1},\delta_r]},(\Omega_{T_{[r^{-1},\delta_r]}}^{\bullet},\mathrm{d}f)))
  \end{equation*}
  is a quasi-isomorphism (in fact, we shall show both are acyclic).

  Below we shall write \(a=r^{-1}\), \(b=\delta_r\). Thus \(0<a<b<1\).
  Let \(U = \mathrm{Sp}(A)\) be an affinoid subdomain of \(T_{[0,b]}\)
  admitting an étale morphism to the disk \(\mathbb{D}^{+}(0;1)^{n}\).
  Then \(W = U \cap T_{[a,b]}\) is an affinoid subdomain of \(U\)
  as \(T_{[a,b]} \to T_{[0,b]}\) is an affinoid morphism.
  It suffices to prove the restrictions of the two Higgs complexes on
  respectively \(U\) and \(V\) are acyclic.

  Write
  \(W=\mathrm{Sp}(B)\).
  Then the morphism \(f\colon \overline{X} \to \mathbb{P}^{1}\)
  gives rise to an element in \(A\), and an element in \(B\) via the morphism \(A \to B\).
  With respect to the the coordinate system provided by the étale morphism
  \(U\to \mathbb{D}^{+}(0;1)^{n}\), the Higgs complex of \(B\) is the Koszul
  complex associated with the partial derivatives
  \(\partial{f}/\partial x_{i}\). Since \(f\) is smooth, these partial
  derivatives form a regular sequence of \(B\) and the Jacobian ideal
  (i.e., the ideal formed by the partials) is not contained in any maximal ideal
  of \(B\). By Nullstellensatz for affinod algebras and standard commutative
  algebra, the Higgs complex is acyclic.

  It remains to show that the ``moderate Higgs complex'' on \(\mathrm{Sp}(A)\)
  is acyclic. With respect to the coordinate system, this is the Koszul complex
  of the partial derivatives of \(f\) on the ring \(A[1/f]\)
  (the localization of \(A\)). Again, since \(f\) is smooth, the partials of
  \(f\) form a regular sequence and has no common zeros. Thus the moderate Higgs
  complex is also acyclic. This completes the proof.
\end{proof}

We leave the apparent generalization to non-proper functions (using logarithmic
forms) to the reader. It should be mentioned that one cannot use algebraic forms
in the non-proper case to calculate the Higgs cohomology, as non-isolated
critical points will contribute infinite dimensional cohomology.

\bibliographystyle{plain}
\bibliography{exp.bib}%
\end{document}